\documentclass[11pt,twoside,english]{article}

\usepackage{latexsym}
\usepackage{amsmath,amsthm}
\usepackage{amsfonts}
\usepackage{amssymb}
\usepackage{graphicx}
\usepackage{color}
\usepackage{epsfig}
\usepackage[modulo,right]{lineno}
\usepackage{hyperref}
\hypersetup{colorlinks=true, citecolor=blue}

\def\leq{\leqslant}
\def\geq{\geqslant}
\def\N{\mathbb{N}}
\def\R{\mathbb{R}}

\def\V{{\bf V}}
\def\W{{\bf W}}
\def\U{{\bf U}}
\def\tv{\tilde{{\bf V}}}
\def\CQFD{\hfill $\Box$}

\newtheorem{Proposition}{Proposition}[section]
\newtheorem{Lemma}[Proposition]{Lemma}
\newtheorem{Definition}[Proposition]{Definition}
\newtheorem{Corollary}[Proposition]{Corollary}

\newtheorem{Remark} [Proposition]{Remark}
\newtheorem{Theorem}{Theorem}

\begin{document}
%\thispagestyle{empty}
%
%{\bf \footnotesize  PREPUBLICACIONES DE MATEM\'{A}TICA}
%
%\vspace{-.3cm}
%
%{\bf \footnotesize \ \ \ \ \ \ \ \ \ \ \ \ \ \ Facultad de Ingenier\'{\i}a }
%
%\vspace{-.3cm}
%
%
% {\bf \footnotesize \ \ \ \ \ Universidad de la Rep\'{u}blica, Uruguay}
%
%\vspace{-.3cm}
%
%
%{\bf \footnotesize \ \ PREMAT 2011/129
%http://premat.fing.edu.uy (2011)}
%
%{\bf \footnotesize Final version  in Journal of Mathematical Biology of SPRINGER. }
%
%{ \bf \footnotesize The final publication will be available at
%http://www.springerlink.com/content/0303-6812  (2012)}
%
%\vspace{.6cm}

\begin{center}
{\LARGE Integrate and Fire Neural Networks,}

{\LARGE Piecewise Contractive Maps and Limit Cycles.}

\vspace{.3cm}
Eleonora Catsigeras\footnote {Instituto de Matem\'{a}tica, Universidad de la
Rep\'{u}blica, Montevideo, Uruguay, 
{\tt eleonora@fing.edu.uy}} \ and Pierre  Guiraud\footnote{Centro de Investigaci\'on y Modelamiento de Fen\'omenos Aleatorios -- Valpara\'iso, Facultad de Ingenier'a, Universidad de Valpara\'{\i}so,  Valpara\'{\i}so,
Chile, {\tt pierre.guiraud@uv.cl}}
\end{center}

%%\linenumbers
%\title{Integrate and Fire Neural Networks, Piecewise Contractive Maps and Limit Cycles}
%\author{E.\ Catsigeras\footnote{Instituto de Matem\'{a}tica, Universidad de la
%Rep\'{u}blica, Montevideo, Uruguay, {\tt eleonora@fing.edu.uy}} \ and P.\ Guiraud\footnote{Departamento de
%Estad\'{\i}stica-CIMFAV, Fac.\ Ciencias, Universidad de Valpara\'{\i}so, Gran Breta\~na 1111, Valpara\'{\i}so,
%Chile, {\tt pierre.guiraud@uv.cl}} } \maketitle

\begin{abstract}
We study the global dynamics of integrate and fire neural networks composed of an arbitrary number of identical neurons interacting by inhibition and excitation. We prove that if the interactions are strong enough, then the support of the stable asymptotic dynamics consists of limit cycles. We also find sufficient conditions for the synchronization of networks containing excitatory neurons. The proofs are based on the analysis of the equivalent dynamics of a piecewise continuous Poincar\'e map associated to the system. We show that for efficient interactions the Poincar\'e map is piecewise contractive. Using this contraction property, we prove that there exist a countable number of limit cycles attracting all the orbits dropping into the stable subset of the phase space. This result applies not only to the Poincar\'e map under study, but also to a wide class of general $n$-dimensional piecewise contractive maps.
\end{abstract}



 \section{Introduction}

Numerous physical or biological systems can be seen as composed of a large number of units in interaction. In many occasions, their time evolution is modeled by a system of coupled differential equations, or by a high dimensional discrete time dynamical system. Those models take into account a proper individual dynamics for each unit and a coupling between units which may depend on the state of the whole system. Typical example of such models are coupled oscillators (continuous time) and coupled map lattices (discrete time) \cite{CF}. They usually assume
well mathematically characterized individual dynamics, the main question under study being how the coupling of units can generate the collective behaviors observed in physical and biological systems. Nevertheless, from a mathematical point of view, these systems are a source of open problems and most mathematical results have been proved under the assumption of weak coupling and$/$or focus on particular solutions \cite{Y08}.

Pulse-coupled oscillators appear frequently in biological sciences, in particular in neuroscience to model neural networks \cite{MB}. In this context, the state of each oscillator describes the difference of electrical potential between the inside and the outside of a neuron's membrane. An archetype of pulse-coupled neural network appears in literature \cite{IZHI99b} in the following form:
\begin{equation}\label{PULSEDCOUPLED}
\dot V_i= f_i(V_i) + \sum_{j = 1}^n   h_{ji}(V_i)\delta(t-t_{j}) \qquad\forall\, i\in\{1,\dots,n\}.
\end{equation}
The solutions of the non-coupled equation $\dot V_i = f_i(V_i)$ define the individual dynamics of the membrane potential of the neuron $i$.
%The function $f_i$ is supposed positive and
The following additional rule is assumed:
if the potential $V_i$ reaches the so called threshold potential $\theta >0$ at an instant denoted $t_i$, then the
neuron $i$ is said to fire (or to emit a spike) and its potential is reset to zero. The term $ h_{ji}(V_i) \delta(t-t_j)$ is a short hand notation meaning that at time $t_j$ the potential of the neuron $i$ suffers a discontinuity jump of amplitude $h_{ji}(V_i(t_j^-))$. This discontinuity $h_{ji}$ is produced
   by the firing of the presynaptic neuron $j$ on the potential $V_i(t_j^-)$ of the postsynaptic neuron $i$. If the jump is negative ($h_{ji}<0$) the interaction is said inhibitory and if it is positive ($h_{ji}>0$) it is said excitatory.

When weak interactions are assumed, and extra conditions on $f_i$ are imposed, it is possible to reduce (\ref{PULSEDCOUPLED}), and also more realistic neural models, to a canonical system of phase coupled oscillators \cite{HI, IZHI99a}. This opens the possibility to get some insight in the dynamics of a huge class of weakly coupled neural networks by studying for example the existence and the stability of synchronized states \cite{E96, HI, IZHI99b}. Further insight in the dynamics of neural networks can be obtained by considering specific models. In this respect (leaky) Integrate and Fire (IF) neural networks \cite{GK02} are certainly the most popular. For these networks $f_i$ is a real affine function. Many mathematical works on IF neural networks deal with the dependence of particular solutions on  the parameters describing the interactions (that are not necessarily impulsive). The effect of the velocity of the interactions on the stability of synchronized and anti-synchronized states is detailed for weak excitatory and inhibitory interactions in the case of two \cite{VAE94, CL97} or more
neurons \cite{V96} and in presence of delay \cite{CL97}, 
\cite{BC00}. Synchronized solutions and more generally phase-locked solutions are also studied in the case of strong coupling for different architectures of network \cite{BC00}.

Although important results about the phenomenology of IF neural networks have been obtained, they principally focus on particular solutions, and there is still few mathematical results about their
{\em global dynamics}, possibly letting unknown important features of neural networks. The purpose of this paper is precisely to give a mathematical description, developing analytical proofs, of
the global dynamics of IF neural networks. In spite of IF neural networks being continuous time dynamical systems, as far as we know, previous studies of their global dynamics develop methods of discrete time dynamical systems. In their seminal work  \cite{MS90}  study a Poincar\'{e} return map to prove that for the system (\ref{PULSEDCOUPLED}) with homogeneous constant excitatory interactions ($h_{ji}=cte>0$) and homogeneous individual dynamics ($f_i=f$), almost all orbits become synchronized. In \cite{C08} a discrete time IF neural network, which may be seen as a discretization of (\ref{PULSEDCOUPLED}) by a formal Euler scheme, is studied. The corresponding dynamical system is defined by the iterations of a general piecewise affine map.
It is proved that for generic values of the parameters, the global asymptotic dynamics is supported on a finite number of stable periodic orbits. It is also proved that for non generic values of the parameters, asymptotic dynamics is sensitive to initial conditions.

Motivated by the rigorous results of \cite{C08}, which are proved for discrete time system,
the question we address in this paper is if they are also true for continuous time IF neural networks. Although elaborated integration schemes, especially designed for the simulation of neural networks have been developed \cite{BRCHBB07}, there is no way to completely eliminate numerical
errors. Thus, the results of simulations can drastically depend on the used integration strategy \cite{RD06}. This motivates the interest of developing rigorous mathematical proofs also
in the continuous time case. In this respect, the previous works \cite{C10, CRB, J02} proved that periodic orbits attract almost all initial conditions, but under the assumption that the interactions are all inhibitory, while arbitrary interactions are considered in \cite{C08}.

The main property allowing the proof of these results is that the (return) map is {\em piecewise contractive} in the whole phase space. In Section \ref{SECCIONMODEL}, we derive the return map of
(\ref{PULSEDCOUPLED}), assuming $f_i(V_i)=-\gamma V_i + K$, the independence of $h_{ji}$ on $V_i$,
and other hypothesis ((H1) and (H2), stated in the same section)  giving a precise meaning to (\ref{PULSEDCOUPLED}).
In Section \ref{SECCIONCONTRACPROP}, we investigate the contraction properties of the return map. We prove that, unexpectedly, it is not piecewise contractive in the whole phase space for an open region of the interactions values, if some of them are excitatory (Theorem \ref{SNOCONT}). Nevertheless, we find also an open subset in the space of parameters such that this return map is piecewise contractive in the whole phase space, with respect to an adapted metric (Theorem \ref{TMETRIC}). This parameters subset is defined by hypothesis (H3) and (H4) stated in Section \ref{SECCIONCONTRACPROP}.

We give now a short version of our results about the global dynamics of IF neural networks:
\begin{Theorem}\label{Theorem1}
1) Under Hypothesis (H1), if the neural network is completely excitatory and the number of neurons is sufficiently large, then all the orbits are eventually periodic and synchronized.

\noindent 2) Under the hypothesis (H1), (H2), (H3) and (H4), if the neural network contains inhibitory neurons, then the stable orbits are attracted by a countable number of limit cycles.
\end{Theorem}
The part 1) of this theorem is reformulated in Theorem \ref{GLOBALSYNCHRO}, which is proved in Section \ref{SYNCHRO}. Its proof is done only under hypothesis (H1), but the hypothesis of existence of a large number of neurons is necessary. In fact, if the neural network is completely excitatory, but the number of neurons is not sufficiently large (in relation with the minimum amplitude of the synaptic interactions), then there exist orbits that are never synchronized (Proposition \ref{NOGLOBALSYNCHRO}).

The part 2) follows from Theorem \ref{teoremaPrincipal} in Section \ref{ASYMPTOTIC}. This last theorem states that the stable asymptotic dynamics of the return map is supported by periodic orbits.
However, it does not prevent the system from exhibiting a weak form of chaos. In particular, stable   chaos 

\cite{PT10} exists because of the coexistence of an uncountable set of sensitive states. The distance between the limit cycles (attracting all the stable dynamics) and the sensitive points influences the period of the cycles and determines the effects of small perturbations (rounding errors, stochastic perturbations). The expected effects are the existence of long transient times and
of cycles of large period, provided the sensitive points and the stable points are sufficiently intricate.

The proof of Theorem \ref{teoremaPrincipal} does not use the particular formulation of the return map. It applies to a wide class of piecewise contractive maps (see Definition \ref{DEFCONTRACTIVEMAP}). Previous results \cite{B06, BD09, CB, C08, GT86}, stating the existence of periodic attractors for piecewise contractive maps, are proved in a different context. In \cite{B06} the maps are one dimensional and injective. In the works studying higher dimensional dynamics, only affine maps \cite{C08, BD09} or injective maps \cite{CB} are considered.  In our case  none of these hypothesis is assumed, since the proof applies to $n$-dimensional maps, that are neither necessarily piecewise affine nor globally injective.

\section{Integrate and fire neural network}\label{SECCIONMODEL}

We propose to study the global dynamics of leaky integrate and fire neural networks. Our working model is a standard IF neural network considering an arbitrary number of neurons connected by inhibitory and excitatory synapses. This system, which is defined precisely in Subsection \ref{secdefmodel}, is the model (\ref{PULSEDCOUPLED}) with $f_i(V_i)=-\gamma V_i + K$ and where $h_{ji}(V_i)=H_{ji}$ is independent of $V_i$.
Its global dynamics is studied in the next sections via a Poincar\'e map which is derived
in Subsection \ref{RETURNMAP}.

\subsection{Definition of the model}\label{secdefmodel}

At each time $t\in\R$, the state of a neuron $i\in I:=\{1,\dots, n\}$ is described by  its membrane potential
$V_{i}(t)$ and the state of the network is represented by the vector ${\bf V}(t)=(V_1(t),\dots,V_n(t))$. According to model (\ref{PULSEDCOUPLED}), the time evolution of the network has two regimes: a sub-threshold regime and a firing regime.

\vspace{1ex}

\noindent  {\bf Equations of the sub-threshold regime:} The sub-threshold regime occurs when $V_i(t)<\theta$ for all $i\in I$, where $\theta>0$ is called the threshold potential. In such a regime, the state of the network satisfies the system of differential equations defined by:
\begin{equation}\label{IF}
\dot V_i(t) = -\gamma V_i(t) + K  \qquad \forall\,i\in I.
\end{equation}
%\begin{equation}\label{IF}
%\dot V_i(t) = -\gamma V_i(t) + K \equiv  -\gamma (V_i(t) - \beta) \qquad \forall\,i\in I, \ \ \ \ \ \ \ \ \ \mbox{where } \beta = \frac{K}{\gamma}.
%\end{equation}
The constant $\gamma>0$ stands for $1/RC$ where $R$ and $C$ are respectively the resistance and the capacity of the neural membrane, and $K=I_{ext}/C>0$ is proportional to a constant external current $I_{ext}$. According to Equation (\ref{IF}), the potential of each neuron tends to the equilibrium value:
\begin{equation}\label{BETA}
\beta:=\frac{K}{\gamma}=RI_{ext}>0.
\end{equation}

\vspace{1ex}

\noindent  {\bf Equations of the firing regime:} If we assume $\beta>\theta$, and take an initial state $\V(0)$ such that $V_i(0) < \theta$ for all $i \in I$, then there exists a smallest time $t_0$ (which depends on $\V(0)$) when the potential of (at least) one neuron reaches the threshold. At this instant the network enters in the firing regime: the neuron emits a spike that induces a change in the potential of all the neurons it is connected with, and its own potential is reset to a smaller value than the threshold, chosen equal to 0 (without loss of generality, and as conventional for integrate and fire models). Therefore, when the network enters in the firing regime,
its state suffers a discontinuity due to the reset of the firing neurons, and to the change of potential of the neurons receiving the spikes of the firing neurons. Formally, if $J\subset I$ denote the set of \em all \em the neurons that reach $\theta$ at time $t_0$, i.e. spontaneously (by the solution flow of equation (\ref{IF})) or by an excitation produced by other neurons, the state of the network satisfies:

\begin{equation}\label{FIRE}
\lim_{t\downarrow t_0} V_i(t)=0 \quad\mbox{if}\quad i\in J \quad\mbox{ and } \quad \lim_{t\downarrow t_0} V_i(t)= \lim_{t\uparrow t_0} V_i(t) + \sum\limits_{j\in J}H_{ji} \quad\mbox{if}\quad i\notin J.
\end{equation}
The constant $H_{ji}$ represents the synaptic interaction triggered by a spike of the neuron $j$ towards the neuron $i$. It is positive for an excitatory synapse, negative for an inhibitory synapse, and equal to 0 if the neurons are not connected. Due to the instantaneous character of the reset and of the synaptic interactions, in presence of excitatory neurons, the dynamics of IF models of type (\ref{PULSEDCOUPLED}) may be ill-defined or exhibit infinite firing rates when the network has some loops\footnote{For example, consider a loop of two neurons where neuron 1 is excitatory, neuron 2 is inhibitory, and the synaptic interactions are sufficiently large to lead to the following situation: the neuron 1 fires instantaneously and induces the instantaneous firing of the neuron 2, but this last firing prevents instantaneously neuron 1 from firing. Then, the state of neuron 1 is undetermined.}, unless some kind of refractory period is considered. To ensure the model is well defined for any network, we suppose that a neuron which fires at time $t_0$, can neither receive nor emit a second spike at time $t_0$ (right side of (\ref{FIRE})). In this sense we consider a refractory phenomenon, which as well as the reset and the synaptic interactions is instantaneous.
%\footnote{\textcolor{blue}{Ok ca demande a etre expliquer, c'est quoi une periode refractaire instantanee. Il faudra peut etre commencer a decrire lavalanche pour comprendre de quoi il s'agit. Par exempel dans un annex raccourci de eleonora. Il faudra dire au referee que le but est d'etudier un model instantanee, que peut etre ca introduit des choses etrange mais cela permet d'obtenir des resultas rigoureux et que on est pas les seuls a faire ca}}
It results that the set $J$ of the neurons that fire at time $t_0$ has to be defined and computed carefully to take into account the refractory phenomenon. In Section \ref{RETURNMAP}, we write the exact mathematical definition of this set (formula (\ref{AVAL}) and (\ref{JAVAL})).

\begin{Definition} \em We will say that a neuron $j$ is excitatory (inhibitory) if all its synapses are excitatory (inhibitory), i.e $H_{ji}\geq 0$ ($H_{ji}\leq 0$) for all  $i\in I$ such that $j \neq i$. We will say that a neuron is ``mixed" if it is neither excitatory nor inhibitory, namely, if it does not satisfy Dale's principle (see \cite{HI} page 7).\em
\end{Definition}

To prove Theorem \ref{Theorem1} and other results along this paper, we will make
the following hypothesis:

\vspace{1ex}
\noindent {\bf (H1)} The membrane potential of the neurons has a lower bound, i.e.
there exists $\alpha<0$ such that for all $i\in I$ and $t\in\R$, we have $V_i(t)\geq\alpha$.

\vspace{1ex}
\noindent {\bf (H2)} If a neuron $i$ suffers inhibitory and excitatory interactions at the same time, and if the sum of the excitatory interactions is  large enough to make it reach the threshold potential, then the neuron $i$ fires.

\vspace{1ex}

The hypothesis (H1) fits with the physical bounds of the electric potentials of real biological
or electronic systems: the potential of a neuron can not be arbitrarily small. We do not specify any particular negative value for $\alpha$. It is a parameter of the model which can change quantitative results, but has little impact on the qualitative dynamics of the network.

%\vspace{1ex}

The hypothesis (H2) is used only to determine the neurons that fire (those of the set $J$ of the equation (\ref{FIRE})), but both excitatory and inhibitory interactions are considered to update the states of the neurons that do not fire (r.h.s. of equation (\ref{FIRE})). Hypothesis (H2) is a technical hypothesis that we need to prove our results, but it solves an indetermination that appears anyway in any network with instantaneous interactions and has a biological interpretation, as
we explain now. In a model for which the interactions are not instantaneous, during a short time interval $\tau$ where the neuron are interacting, synaptic weights of different signs are added to the potential of the post--synaptic neuron in a certain order. If the sum of the excitatory weights is larger than the threshold but the total sum of synaptic weight is smaller than the threshold, in the time interval $\tau$, the following two situations can arise:

\vspace{1ex}
\noindent 1) Excitatory signals arrive first to a neuron $i$, and their sum is large enough to make it fire. Then, neuron $i$ fires, and besides, due to the refractory period, the sum of the inhibitory signals that arrive delayed to neuron $i$ does not change its potential $V_i\simeq 0$.

\vspace{1ex}
\noindent 2) The same excitatory and inhibitory signals that in case 1) arrive to neuron $i$, but in the reverse order. Then $V_i$ initially decreases and when the positive excitatory signals arrive delayed to the neuron $i$ they are not enough to make it fire.

\vspace{1ex}
\noindent When the interactions are instantaneous, it is possible that excitatory and inhibitory signals arrive to neuron $i$ at the same time, since the time interval $\tau$ while the neurons are interacting is collapsed to zero. The previous example shows that the algebraic sum of the excitations and the inhibitions received during the time interval $\tau$ is not necessarily a realistic criterion to decide if the neuron $i$ fires or not. Hypothesis (H2) is an alternative criterion, which coincides with the one of the algebraic sum in most situations, but assumes that the positive interactions act faster than the negative ones when an indetermination exists. It can also be seen as a way to take into account that excitatory synapsis are more frequent than inhibitory synapsis in some part of the nervous system \cite{MEA01}.

\subsection{The Poincar\'e return map}\label{RETURNMAP}

In order to analyze the global dynamics of the IF neural network, we reduce it to an equivalent discrete time dynamical system, namely a Poincar\'e return map. In this section, we introduce all the important notions we use in the sequel of the paper. Among then we define a Poincar\'e section in the phase space, we compute the waiting time before the spontaneous firing of the network, we explain in details the rules of the firing regime, we introduce a partition of the Poincar\'e section which atoms are the sets of the initial conditions leading the same neurons to fire in the firing regime, and we derive a Poincar\'e return map of the network.

\vspace{1ex}
\noindent{\bf Poincar\'e section $\Sigma$:} Since the potential of a neuron is always larger than $\alpha$ and always smaller than $\theta$, the states of the
network always belong to the $n$-dimensional space $Q={[\alpha, \theta]}^n$. By definition of the model, the
network never stops to emit spikes (since $\beta>\theta$). It exists then arbitrarily large times such that the
potential of a neuron is reset to zero. In other words, any solution of the model returns infinitely many times to
the set:
\begin{equation}\label{PARTITIONBJ}
\Sigma=\bigcup_{j=1}^{n}{\hat \Sigma}_j \quad\mbox{ where }\quad
{\hat \Sigma}_j=\{{\bf V}\in Q\ :\ V_j=0\}.
\end{equation}
The set $\Sigma$ is the Poincar\'{e} section that we will consider. The topology we use is the one induced by the
embedding $\Sigma\subset\R^n$. Specifically, we consider in $\Sigma$ the metric derived from the supremun norm of $
\R^n$, denoted $\|\cdot\|$ in this paper and defined by $\|\V\|=\max\limits_{i\in I}|V_i|$.

\vspace{1ex}
To sum up the principal steps of the construction of the return map, let us follow an orbit of the network. Suppose the initial state of the network is $\V\in\Sigma$. If $\V$ is such that $V_i<\theta$ for all $i\in\{1,\dots,n\}$, then the network is in the sub-threshold regime and there is a waiting time ${\bar t}(\V)>0$ before a neuron  reaches the threshold potential. At the instant ${\bar t}(\V)$ the network enters in the firing regime, and a set $J(\V)$ of neurons emit some spikes. The potential of these neurons is then reset to $0$ and the states of the other neurons that did not spike during the firing regime, because excitations were not enough, is updated according to the interactions received. The network is back in the sub-threshold regime in a point $\rho(\V)$ of $\Sigma$, which is the value at point $\V$ of the return map $\rho$ we want to construct. The formula of the return map is the same for all $\V$ sharing a same $J(\V)$, this why we will introduce a partition $\mathcal{P}$ of $\Sigma$ which atoms are precisely the points with a same $J(\V)$. Now we detailed the computation of ${\bar t}(\V)>0$, $J(\V)$, $\mathcal{P}$ and $\rho$.

\vspace{1ex}
\noindent {\bf Waiting time ${\bar t}({\bf V})$:} Solving the system (\ref{IF}) leads to the time $t$ map $\phi^t=(\phi^t_1,\dots,\phi^t_n)$ where for each $i\in I$ and $t\in\R$
\begin{equation}\label{FLOW}
\phi^t_i({\bf V})=(V_i-\beta)e^{-\gamma t} + \beta\quad \forall\,{\bf V}\in\R^n.
\end{equation}
Thus, if at time $t=0$ the network is in the state ${\bf V}\in \Sigma$, it enters in the firing regime at time:
\begin{equation}\label{WTIME}
{\bar t}({\bf V}):=\min\limits_{i\in\{1,\dots,n\}}t_i({\bf V}) \quad
\mbox{where}\quad
 t_i({\bf V}):=\inf\{t\geq 0\ :\ \phi_i^t({\bf V})\geq\theta\}.
\end{equation}
Note that ${\bar t}({\bf V})$ is the value of $t_0$ of formula (\ref{FIRE}) when the initial state of the network is $\V$. In the particular case where $\V\in\Sigma$ is such that $V_i=\theta$ for some $i$, the waiting time is equal to $0$, if not it is positive. Note also that if ${\bar t}({\bf V})=t_i(\V)$ the expression of the time $t$ map a time ${\bar t}({\bf V})$ is given by
\begin{equation}\label{PHITV}
\phi_i^{{\bar t}({\bf V})}(\V)=\theta \quad\text{and}\quad \phi_k^{{\bar t}({\bf V})}(\V)=\beta-\frac{(\beta-V_k)(\beta-\theta)}{(\beta-V_i)} \quad\forall\,k\neq i.
\end{equation}
We will often use this expression latter on.
%From (\ref{FLOW}) we deduce that the orbit of $\V$ under the dynamics of the model satisfies:
%\[
%\V(t)=\phi^t({\bf V}) \quad\forall\, t\in[0,{\bar t}({\bf V}))\quad\text{and}\quad
%\lim_{t\uparrow {\bar t}({\bf V})}\V(t)=\phi^{{\bar t}({\bf V})}({\bf V}).
%\]

\vspace{1ex}
\noindent {\bf Firing regime $t={\bar t}({\bf V})$:} By definition (\ref{WTIME}), at time ${\bar t}({\bf V})$ some component of the vector $\phi^{{\bar t}({\bf V})}({\bf V})$ are equal to $\theta$; those corresponding to the neurons which potential reaches the threshold in a smaller time than the potential of all the other neurons. We need to characterize this set of neurons that fire {\it spontaneously} because they are part of the set $J$ of formula (\ref{FIRE}). It depends of course of the value of $\V$ (as well as $J$). This is why we introduce the following cover of $\Sigma$:

\begin{equation}\label{PARTITIONBI}
\Sigma = \bigcup _{i=1}^n \Sigma_i \quad \mbox{ where } \quad
\Sigma_i =\{{\bf V}\in \Sigma\ :\ \overline{t}({\bf V})= t_i({\bf
V})\}.
\end{equation}
A set $\Sigma_i$ of this cover, is the set of the initial states in $\Sigma$ such that
the  neuron $i$ fires spontaneously after the waiting time $\overline t(\V)$, that is,  if $\V\in\Sigma_i$ then $\phi_i^{{\bar t}({\bf V})}({\bf V})=\theta$.
Since several neurons may reach $\theta$ at the same time, these sets are not pairwise disjoint (and thus forms a cover instead of a partition of $\Sigma$), but they have pairwise disjoint interiors. We can states any way
that the neurons that fires spontaneously at time $\overline t(\V)$ are those of the set of indexes:
\[
J_0({\bf V})=\{i\in I\,:\, {\bf V}\in \Sigma_i\}.
\]

Now, the firing of the neurons of $J_0({\bf V})$ may instantaneously excite other neurons, which may also fire in turn at the same instant, and excite other neurons or not. One more time, we need to characterize all the neurons participating in this instantaneous \em avalanche \em process, i.e. the set $J$ of formula (\ref{FIRE}), that
we denote until now $J(\V)$. We can determine $J(\V)$ introducing a recursive sequence of sets of indexes
$\{J_m(\V)\}_{m \geq 0}$, where each $J_m(\V)$ is the set of neurons that have fired once until the step $m$ of the avalanche process. Taking into account the temporary inertia (the refractory period) of real neurons after firing, and considering Hypothesis (H2), this sequence must obey the following induction rule for all $m\geq 1$:
\begin{equation}\label{AVAL}
J_{m}({\bf V})=J_{m-1}({\bf V})\cup\{k\in I\setminus J_{m-1}({\bf V})\,: \, \phi_k^{\overline{t}({\bf V})}({\bf
V}) +\sum_{i\in J_{m-1}({\bf V})\,:\, H_{ik}>0} H_{ik}\geq\theta\}.
\end{equation}
As an example, $J_{1}({\bf V})$ contains the neurons of $J_0(\V)$ and possibly additional neurons that fire because of some excitatory interactions with the neurons of $J_0(\V)$. Note that this additional neurons, which are those of the second set of the union that defines $J_1({\bf V})$, cannot be neurons of $J_0(\V)$. In our definition (\ref{AVAL}), they are not considered as a possible receptors of spikes, in order to take into account of a \em refractory phenomenon \em in the model. Indeed, a set $J_m(\V)$ is the disjoint union of the neurons of $J_{m-1}(\V)$ with \em new \em neurons firing by interactions with those of $J_{m-1}(\V)$.

Since the set $I$ is finite, there exists a $m_0\geq 0$, such that no new neurons are incremented to $J_{m_0+1}(\V)$ with respect to $J_{m_0}(\V)$, either because $J_{m_0-1}(\V)$ is the whole all set $I$, or because the sum of the excitatory interactions of the neurons of $J_{m_0-1}(\V)$ is not enough to make fire some new neurons. For such a $m_0$ we have $J_{m_0}(\V)=J_{m_0-1}(\V)$
and we can conclude that \em the set of the neurons that emit a spike at time ${\bar t}({\bf V})$ \em can be written as:
\begin{equation}\label{JAVAL}
J({\bf V})=\bigcup_{m\in\N} J_{m}({\bf V})=J_{m_0}(\V).
\end{equation}
If there is no excitatory neurons in $J_0({\bf V})$, then $J_m(\V)=J_0(\V)$ for all $m\in\N$ and $J(\V)=J_0(\V)$. If $J_0({\bf V})$ contains excitatory neurons, then $J(\V)$ may contain more neurons than $J_0(\V)$.

\vspace{1ex}
\noindent{\bf Partition $\mathcal{P}$ of $\Sigma$:} As shown in the previous construction, to each initial state of the network $\V\in\Sigma$ it is associated the set $J(\V)$ of the neurons that fire at time ${\bar t}(\V)$. So we can consider the function $J(\cdot)$ of $\Sigma$ into the set $P(I)$ of all the nonempty subsets of $I$ and use its pre-images to get a partition $\mathcal{P}$ of $\Sigma$:
\begin{equation}\label{PARTITION}
\mathcal{P}={\{\Sigma_{J}\}}_{J\in P(I)}\quad\mbox{ where }\quad
\Sigma_J=\{{\bf V}\in \Sigma \,:\, J({\bf V})=J\}.
\end{equation}
Given a set of neurons $J\in P(I)$, the set $\Sigma_J$ is the set of all the initial states in $\Sigma$ such that the neurons that fire after the waiting time $\overline t(\V)$ are exactly those of $J$. For example, if $\V\in\Sigma_{\{i\}}$, then at time $\overline t(\V)$, only the neuron $i$ fires (which makes a difference with the set $\Sigma_i$ of the cover (\ref{PARTITIONBI})). Note that ${\bar t}(\V)$ is not necessarily the same for each $\V$ in a same $\Sigma_J$. It is straightforward to check that $\mathcal{P}$ forms a partition of $\Sigma$: on one hand, to each $\V\in\Sigma$ it is associated some $J(\V)$ and thus some $\Sigma_J$, and on the other hand, if $\V\in\Sigma_J\cap\Sigma_{J'}\neq\emptyset$ then by definition $J=J(\V)=J'$ and $\Sigma_J=\Sigma_{J'}$. According to the values of the parameters of the model $H_{ji},\gamma, \beta$ and $\alpha$, some $\Sigma_J$ may be empty. Nevertheless, there is always at least one nonempty  set $\Sigma_J$ in the partition ${\mathcal P}$.

\vspace{1ex}
Now we can give the formula of a return map in $\Sigma$:

\begin{Proposition} Under Hypothesis (H1) and Hypothesis (H2), the map $\rho:\Sigma\to\Sigma$ which components $\rho_1\dots,\rho_n$ are defined in each atom $\Sigma_J$ of the partition $\mathcal{P}$ of $\Sigma$ by:
\begin{equation}\label{RETURN}
\rho_i({\bf V})= 0 \quad\mbox{if}\quad i\in J \quad\mbox{and}\quad \rho_i({\bf V})=\max\{\alpha,\,\phi_i^{{\bar
t}({\bf V})}({\bf V})+\sum\limits_{j\in J} H_{ji}\} \quad\mbox{if}\quad i\notin J
\end{equation}
is a return map in $\Sigma$ of the model.
\end{Proposition}

\begin{proof} Take $\V\in\Sigma$, and let $J\in P(I)$ such that $\V\in\Sigma_J$. From (\ref{FLOW}) and (\ref{WTIME}) we deduce that the orbit of $\V$ under the dynamics of the model satisfies:
\begin{equation}\label{RECONS}
\V(t)=\phi^t({\bf V}) \quad\forall\, t\in[0,{\bar t}({\bf V}))\quad\text{and}\quad
\lim_{t\uparrow {\bar t}({\bf V})}\V(t)=\phi^{{\bar t}({\bf V})}({\bf V}).
\end{equation}
At time ${\bar t}(\V)$ the network enter in the firing regime and the neurons that fires are those of $J$, since $\V\in\Sigma_J$. According to (\ref{FIRE}), the potential of the neurons of $J$ is reset and the potential of the other neurons suffers excitatory and inhibitory interactions from those of $J$. Together with Hypothesis (H1) it implies:
\begin{equation*}
\lim_{t\downarrow {\bar t}({\bf V})} V_i(t)=0 \quad \mbox{if}\quad i\in J \quad\mbox{and} \quad \lim_{t\downarrow {\bar t}({\bf V})} V_i(t)=\max\{\alpha, \phi_i^{{\bar t}({\bf V})}({\bf V}) + \sum\limits_{j\in J}H_{ji}\} \quad\mbox{if}\quad i\notin J.
\end{equation*}
Thus, $\rho(\V):=\lim_{t\downarrow {\bar t}({\bf V})}\V(t)$ is a point of $\Sigma$, which is the new state of the network when it comes back in the sub-threshold regime. \qed
\end{proof}

\vspace{1ex}
An orbits $\{\rho^n(\V)\}_{\{n\in\N\}}$ of the return map gives the states of the network immediately after each spike (or simultaneous group of spikes). However, the entire orbit of the network can be reconstructed from the orbit by the return map using (\ref{RECONS}), and all the properties of the network can be deduced from those of the return map. In particular, the network has a periodic orbit if and only if the return map also has.

The return map does not satisfy standard hypothesis of the dynamical systems theory, such as differentiability, or continuity in its entire domain. Actually, $\rho$ is continuous in the interior of each set $\Sigma_J$ of the partition ${\mathcal P}$, but not necessarily in their union $\Sigma$. At a point in the boundary $\partial\Sigma_{J}$ of a set $\Sigma_J$, generically the return map is  not continuous, since a small perturbation may belong to the interior of $\Sigma_J$ or to another set $\Sigma_{J'}$ and thus it can change the set of the firing neurons. A detailed study of these properties is given in Section \ref{subsectionPartitionSubc}.

In general, it is not trivial to characterize the atoms of $\mathcal{P}$ because they depend strongly on the interactions. However, if a network is completely inhibitory, when a neuron fires, it does it spontaneously. This makes easier the computation of $\mathcal{P}$, since it allows to replace $J(\V)$ by $J_0(\V)$ in the definition (\ref{PARTITION}), and this last set does not depend on the interactions but only on the  characteristics of the sub-threshold regime. For such networks, it is then possible to show that the boundary of $\mathcal{P}$  are the sets $\Sigma_{J}$ such that $J$ contains more than one neuron. It follows that the discontinuity points of the return map are the initial states leading two neurons or more to fire together spontaneously.

\section{Contraction properties of the return map}\label{SECCIONCONTRACPROP}

In several mathematical studies of the global dynamics of IF neural networks, the considered model  has the property to be {\it piecewise contractive} in the {\it whole} phase space (see Definition \ref{CONTRACTIVEZONE}). This property reflects the presence of dissipation in the networks. It is introduced in the model to that aim \cite{C08}, or it is a consequence of the absence of excitatory interactions \cite{C10, CRB} or of the predomimance of the inhibitory interactions \cite{J02}.
In our case, where none of these hypothesis is assumed, it is not a property of the return map that holds  generically in the space of the parameters of the system. Indeed, in Theorem \ref{SNOCONT}, we prove that there exist open sets in the space of the parameters such that the return map is not piecewise contractive in a subregion of the phase space. This result holds for any metric derived from a norm and in this sense is an inherent property of the system.

Nevertheless, we also show in Proposition \ref{CONTRACTIVEZONE} that it always (co)exists a subregion of the phase space where the return map is piecewise contractive. Moreover, the size of this region increases when the external current $I_{ext}$ decreases.
% in such a way that $\beta$ is not too far from the threshold potential $\theta$.
By Theorem \ref{SNOCONT}, a contractive region cannot coincide with the whole phase space for any values of the interactions. However, we show in Theorem \ref{TMETRIC} that if the interactions are not too small, or equivalently the external current is not too strong, then it is possible to find a metric for which the return map is piecewise contracting in the whole phase space, provided the network has not only excitatory neurons.

%Nevertheless, we  also show in Proposition \ref{CONTRACTIVEZONE}  that if $\beta$ is not too far from the threshold potential $\theta$, it always (co)exists a subregion of the phase space where the return map is piecewise contractive. By Theorem \ref{SNOCONT}, a contractive region cannot coincide with the whole phase space for any values of the interactions. However, we show in Theorem \ref{TMETRIC} that if the interactions are not too small, and if the network is not only composed of excitatory neurons, then it is possible to find a metric for which the return map is piecewise contracting in the whole phase space.

%En gros il faut travailler pour ce ramener a une situations ou il piecewise contracting dans tout l'espace de phase
% Cette exemple montre qu'il est nŽcessaire de faire des hypothses pour assurer l'existence
% Que c'est pas trivialement vraie.

%sous la conditions que le network ne contient pas seulement des neurones excitateurs. J'ai envie de dire que ici il faut travailler, pour etudier si il y a contraction ou non et c'est ce qu'on fait ici. Cette etude montre que, pour certaine valeur des interactions il existe une toujours une zone expansive. Cependant des conditions sur beta et theta garantise l'existene d'une zone contractive quelles que soit les interactions, mais qui n'est pas tout l'espace de phase. Dans un second temps en mettant deplus des conditions raisonables sur les interactions on peut demontrer l'existence d'une metrique qui n'est pas l'original mais qui garantie que le mapping et contractant dans tout l'espace de phase.

\begin{Definition}\label{CONTRACTIVEZONEDEF} \em We say that a subset $\mathcal{C}$
of the Poincar\'{e} section $\Sigma$ is a contractive zone for the return map $\rho$,
if there exist a finite partition $\mathcal{P}_c$ of $\mathcal{C}$, a constant $0<\lambda_c<1$, and
a norm $\|\cdot\|_c$ such that
\begin{equation}\label{INEGALITECONTRACT}
\|\rho({\bf V})-\rho({\bf W})\|_c\leq\lambda_c\|{\bf
V}-{\bf W}\|_c
\end{equation}
for all ${\bf V}$ and ${\bf W}$ in a same piece of $\mathcal{P}_c$. If there exists a forward invariant  contractive zone, that is $\rho(\mathcal{C})\subset\mathcal{C}$, then $\rho$ is said to be piecewise
contractive in $\mathcal{C}$. \em
\end{Definition}

%Note that the property of being piecewise contracting depends on the partition and on
%the elected norm. Thus it can exists some partitions and norm such that (\ref{INEGALITECONTRACT}) is not
%satisfied, even if for other choices $\rho$ admits a contractive zone. But, if $\rho$ is piecewise contractive in $\mathcal{C}$ with respect to $\mathcal{P}_c$, then $\rho$ is Lipschitz continuous in
%each piece of $\mathcal{P}_c$. Thus, $\mathcal{P}^{\ast}$ has to be a
%refined partition of the natural one $\mathcal{P}$, defined in Equality (\ref{PARTITION}).

%\begin{Definition}\label{CONTRACTIVEZONEDEF} \em Given a subset $\Sigma^{\ast}_c$   { of the Poincar\'{e}
%section } $\Sigma$ and a finite partition $\mathcal{P}^{\ast}$ of $\Sigma^{\ast}_c$, we
%say that $\rho$ is piecewise contractive in $\Sigma^{\ast}_c$ with
%respect to $\mathcal{P}^{\ast}$ if there exist a constant
%$0<\lambda^{\ast}<1$ and a norm $\|\cdot\|_c$ such that
%\[
%\|\rho({\bf V})-\rho({\bf W})\|_c\leq\lambda^{\ast}\|{\bf
%V}-{\bf W}\|_c
%\]
%for all ${\bf V}$ and ${\bf W}$ in a same piece of
%$\mathcal{P}^{\ast}$. In such a case, we say that
%$\Sigma_c^{\ast}$ is a contractive zone with respect to
%$\mathcal{P}^{\ast}$.\em
%\end{Definition}

\subsection{Expansion and contraction in the Poincar\'e section}\label{subsectionSNOCONT}

The question we address now, is the existence of a norm
$\|\cdot\|_c$ such that the whole Poincar\'e section $\Sigma$
would be a contractive zone with respect to the natural partition
$\mathcal{P}$ defined by the equality (\ref{PARTITION}).
The following theorem shows that for some networks such a norm does not exist.
%the return map $\rho$ is not
%always piecewise contractive because this property depends on the values of the interactions.

\begin{Theorem}\label{SNOCONT} Under the hypothesis (H1) and  (H2), there exists an open region of the
values of the interactions such that, for any  norm $\|\cdot\|_c$,
the Poincar\'{e} section $\Sigma$ is not a contractive zone with respect to
the natural partition  $\mathcal{P}$.
\end{Theorem}
In order to construct a subregion of the Poincar\'e section where the return map fails to be piecewise contractive, let us define for all $i\neq j\in I$ the following sets:
\[
\Gamma_{i}:=\{\V\in \Sigma\ :\ c^* < V_i< \theta,\ V_k=0\ \forall\,k\neq i\}\quad\text{where}\quad c^{\ast}:=\beta-\sqrt{\beta(\beta-\theta)}.
\]
Note that for any $0<\theta<\beta$ we have $\beta-\theta<\sqrt{\beta(\beta-\theta)}<\beta-\theta/2$
which implies that  $\theta/2 < c^* < \theta$. So, a set $\Gamma_i$ consists of the initial states of the network such that all the neurons have their potential equal to zero, except the neuron $i$, whose potential is bigger than the quantity $c^\ast>\theta/2$ and smaller than $\theta$. Therefore, for all $\V\in\Gamma_i$ the neuron $i$ reaches the threshold before the other ones, that is ${\bar t}(\V)=t_i(\V)$, and is the neuron that triggers the firing regime.

\vspace{1ex}
Now we state and prove the following lemma establishing a key property to prove Theorem \ref{SNOCONT}.

\begin{Lemma}\label{LEMDILAT2} Let $i\in I$ and suppose $\Gamma_i\subset\Sigma_J$ for some $J\in P(I)$, then for all
$ \V\neq\W\in\Gamma_i$ such that $\rho(\V)$ and $\rho(\W)$ are bigger that $\alpha$, we have
\[
\quad|\rho_k(\V)-\rho_k(\W)|>|V_i-W_i|\quad\text{if}\quad k\notin J,
\]
and moreover $\|\rho(\V)-\rho(\W)\|>\|\V-\W\|$.

\end{Lemma}
\begin{proof} Let us compute $\rho(\V)$. On one hand, since $\V\in\Sigma_J$, we have $\rho_k(\V)=0$ for all $k\in J$. On the other hand, since $\V\in\Gamma_{i}$, we have $V_i>V_k=0$ for all $k\neq i$ which implies that ${\bar t}(\V)=t_i(\V)$. From the formula of the return map, it follows that for all $k\notin J$
\[
\rho_k(\V)=\phi_k^{t_i(\V)}(\V)+\sum\limits_{j\in J} H_{jk}=\beta-\frac{\beta(\beta-\theta)}{\beta-V_i}+
\sum\limits_{j\in J} H_{jk},
\]
where the second equality is obtained applying (\ref{PHITV}) with $V_k=0$.
The same computation being true for $\rho(\W)$ we deduce that for all $k\notin J$
\begin{align}
|\rho_k({\bf V})-\rho_k({\bf W})| &= \left|\phi_k^{t_i(\V)}(\V) - \phi_k^{t_i(\W)}(\W)\right| \nonumber\\
                                  &= \left|\frac{\beta(\beta-\theta)}{(\beta-W_i)(\beta-V_i)}(W_i-V_i)\right|
      > \frac{\beta(\beta-\theta)}{(\beta-c^{\ast})^2}|V_i-W_i|=|V_i-W_i|.\label{ISIDIF}
\end{align}
Now, $\|\rho(\V)-\rho(\W)\|=\max\{|\rho_k({\bf V})-\rho_k({\bf W})|,\ k\in I\}>|V_i-W_i|=\|\V-\W\|$, which completes the proof. \qed
\end{proof}

\noindent Actually, the hypothesis $\Gamma_i\subset\Sigma_J$ for some $J\in P(I)$ can be replaced by the weaker hypothesis: $\Gamma_i\cap\Sigma_J$ is not empty and is not a singleton, for some $J\in P(I)$. Then, the results of the lemma remain true for all $\V\neq\W\in\Gamma_i\cap\Sigma_J$. This weaker hypothesis always holds as $\Gamma_i$ is infinite and the partition $\mathcal{P}$ of $\Sigma$ is finite. This observation implies that the return map always exhibits expansion in some part of its phase space, independently of the values of the interactions (which determine the sets $\Sigma_J$), provided they are enough to ensure that $\rho(\V)$ and $\rho(\W)$ are bigger than $\alpha$.

\vspace{1ex}
Lemma \ref{LEMDILAT2} is stated for the distance induced by the supremum norm, and a priori, it may exist another norm for which the contraction property holds in the whole phase space, even in the $\Gamma_i$. Now, we are going to show that at least for some networks such a norm cannot exit. Examples of such networks have two excitatory neurons $i\neq j \in I$ verifying simultaneously the three following open conditions (O1), (O2), and (O3) :

\begin{itemize}
\item (O1) The neurons $i$ and $j$ cannot be strongly excited by the sequel of the network:
\[
 \sum\limits_{l\neq s\,:\,H_{ls}>0} H_{ls}<\theta-c^\ast \quad\forall\, s\in\{i,j\},
\]
\item (O2) A spike from $i$ or $j$ to any other neuron of the network produces a strong excitation:
\[
 H_{sk}>\theta  \quad\forall\,s\in\{i,j\} \quad\text{and}\quad k\notin\{i,j\},
\]
\item (O3) The sum of the negative and positive interactions that $i$ and $j$ can receive is positive:
\[
 \sum\limits_{l\neq s} H_{ls}>0 \quad\forall\, s\in\{i,j\}.
\]
\end{itemize}
For sake of simplicity we make the proof of Theorem \ref{SNOCONT} with this set of parameters, but other sets can be chosen. In particular, the condition (O2) can be replaced by a more biologically realistic condition, which does not impose some interactions to be bigger than the threshold potential, see Remark \ref{OTHERSETS}. The important point for our proof to work, is to ensure that $\rho^p(\Gamma_i)$ intersects $\Gamma_i$ for some $p$ and that the dilatation is conserved a least for some couple of points during the first $p^\text{th}$ iterations. The advantage of the proposed set of parameters is that it permits to show intuitively that this property is fulfilled for $p=2$, as shown in the following.

\vspace{1ex}
Let us consider a network verifying (O1) to (O3) and whose state $\V\in\Gamma_i$. Then, by definition of $\Gamma_i$, the neuron $i$ fires spontaneously at time ${\bar t}(\V)=t_i(\V)$. At this instant, by condition (O2), the firing of $i$ induces the simultaneous firing of all the other neurons, expect $j$ which is not sufficiently excited (condition (O1)). Thus, the set of the neurons that fire at time ${\bar t}(\V)$ is $J(\V)=I\setminus\{j\}$ and  it follows that: $\V\in\Sigma_{I\setminus\{j\}}$ and $\rho_k(\V)=0$ for all $k\neq j$. Now, if ${\bar t}(\V)$ is sufficiently large, that is if the initial potential $V_i$ of the neuron $i$ is sufficiently near $c^\ast$, then, just before the firing regime, the potential of the neuron $j$ is near $c^\ast$ (actually it is equal to $c^\ast$ is $V_i=c^\ast$). Thanks to condition (O3), the interactions of the neuron $j$ with the network during the firing regime, help its potential to become larger than $c^\ast$. In other words, $\rho_j(\V)\in(c^\ast,\theta)$ and $\rho(\V)\in\Gamma_j$.
We can conclude that for the considered network  $\rho(\Gamma_i)\cap\Gamma_j\neq\emptyset$. These results are stated in Lemma \ref{PARAM} and their rigorous proof is detailed in appendix \ref{PROOFPARAM}.

\begin{Lemma}\label{PARAM} Let $i\neq j\in I$. For the open region of values of the interactions that satisfy the conditions (O1) to (O3), the set $\Gamma_{i}$ is a subset of $\Sigma_{I\setminus\{j\}}$, the set  $\Gamma_{j}$
is a subset of $\Sigma_{I\setminus\{i\}}$, and there exists $(a,b)\subset (c^\ast,\theta)$ such that
for all $\V\in\Gamma_{i}$ satisfying $V_i\in(a,b)$, we have $\rho(\V)\in \Gamma_{j}$. Moreover, if $\W\neq\V\in\Gamma_{i}$ then $\rho(\V)\neq\rho(\W)$.
\end{Lemma}

\begin{proof} See the appendix \ref{PROOFPARAM}. \qed
\end{proof}
According to Lemma \ref{PARAM}, there exist $\V\neq\W\in\Gamma_i$ such that $\rho(\V)\neq\rho(\W)\in\Gamma_j$.
Applying Lemma \ref{LEMDILAT2} to $\rho(\V)$ and $\rho(\W)$ we obtain:
\[
\quad|\rho_i(\rho(\V))-\rho_i(\rho(\W))|>|\rho_j(\V)-\rho_j(\W)|
\]
since $\Gamma_j\subset\Sigma_{I\setminus\{i\}}$. Now, applying once again Lemma \ref{LEMDILAT2}, but this time to $\V$ and $\W$ we obtain:
\[
\quad|\rho_j(\V)-\rho_j(\W)|>|\V_i-\W_i|
\]
since $\Gamma_i\subset\Sigma_{I\setminus\{j\}}$. So we have,
\begin{equation}\label{DILAT2}
|\rho_i^2(\V)-\rho_i^2(\W)|>|V_i-W_i|.
\end{equation}

\begin{proof}{\em of Theorem \ref{SNOCONT}\em} Let us show that (\ref{DILAT2}) is incompatible with
the existence of a norm $\|\cdot\|_c$ such that $\Sigma$ is a contractive zone.
If we suppose such a norm exists, according to Definition \ref{CONTRACTIVEZONEDEF}, for the $\V$ and $\W$ of (\ref{DILAT2}) we have
\begin{equation}\label{CONTRACT2}
\|\rho^{2}(\V)-\rho^{2}(\W)\|_c\leq{\lambda_c}\|\rho(\V)-\rho(\W)\|_c\leq{\lambda}_c^{2}\|\V-\W\|_c<\|\V-\W\|_c.
\end{equation}
%since $\rho(\V)$ and $\rho(\W)$ belong to the same atom of $\mathcal{P}$, and the same goes for $\V$ and $\W$.

Now consider the restriction of the norm $\|\cdot\|_c$ to the set $G_i:=\{0\}^{i-1}\times\R\times\{0\}^{n-i}$, that is the norm $\|\cdot\|_{c,i}$ defined by $\|\U\|_{c,i}:=\|\U\|_c$ for all $\U\in G_i$.
As $\|\cdot\|_{c,i}$ is a norm in a vector space isomorphic to $\R$, there exists $\mu_i>0$ independent of
$\U\in G_i$ such that $\|\U\|_{c,i}=\mu_i|U_i|$. Therefore, for all $\U\in G_i$, we have $\|\U\|_c=\|\U\|_{c,i}=\mu_i|U_i|$.
As $\rho^2(\V)$ and $\rho^2(\W)$ belong to $G_i$ (recall that $\rho(\V),\rho(\W)\in\Gamma_j\subset\Sigma_{I\setminus\{i\}}$) and  as $\V$ and $\W$ belong to $\Gamma_i\subset G_i$, the inequality (\ref{CONTRACT2}) can be written:
\begin{equation}\label{CONTRATH}
\|\rho^{2}(\V)-\rho^{2}(\W)\|_c=\mu_i|\rho_i^{2}(\V)-\rho_i^{2}(\W)|<\mu_i|V_i-W_i|=\|\V-\W\|_c
\end{equation}
which contradicts (\ref{DILAT2}). \qed
\end{proof}

\begin{Remark}\label{OTHERSETS}\em
%We proved Theorem \ref{SNOCONT} using the regions $\Gamma_i$ of the phase space and a particular set of parameters defined by the condition (O1) to (O3). This choice avoid the use of more complicated notations, but other open sets of parameters, including more biologically realistic ones, also work.

1) If we consider the sets $\Gamma_{i,d}:=\{\V\in \Sigma\ :\ c^* < V_i< d,\ V_k=0\ \forall\,k\neq i\}$ where $d<\theta$, the condition (O2) can be relaxed, without changing the important lines of the proof of Theorem \ref{SNOCONT}. Indeed, the condition (O2) ensures that when the initial state of the network belongs to $\Gamma_i$ the firing of the neuron $i$ induces the firing of any neuron $k\notin\{i,j\}$. But if $\V$ is in $\Gamma_{i,d}$ and $d$ is near from $c^\ast$, the waiting time before the firing of the neuron $i$ will be large enough for the potential of any neuron $k\neq i$ to be near from $c^\ast$. In the limit $d=c^\ast$, if $\V\in\Gamma_{i,d}$, an excitation $H_{ik}>\theta-c^\ast$ from $i$ to $k$ is enough to make $k$ fire (use  (\ref{PHITV}) with $V_i=c^\ast$ and $V_k=0$). Thus, the condition (O2) can be replaced by  the more realistic condition $H_{sk}>\delta(d)$ for all $k\notin\{i,j\}$ and $s\in\{i,j\}$, where $\delta(\cdot)$ is a decreasing function of $d$ such that $\lim_{d\to c^\ast}\delta(d)=\theta-c^\ast$. This former limit is always smaller that $\theta/2$ and tends to $0$ when $c^\ast$ goes to $\theta$ (that is when $\beta$ goes to $\theta$, which implies that the waiting time tends to infinity).

\vspace{1ex}
\noindent 2) On the other hand, completely different sets of parameters can be used to prove Theorem \ref{SNOCONT} using the same sets $\Gamma_i$ and a similar proof. Nevertheless, in order to be able to compare different norms, and to obtain the final contradiction (\ref{CONTRATH}), the set $\Gamma_i$ must be subsets of $G_i$ and the parameters must ensure the strong restriction $\rho^2(\Gamma_i)\cap G_i\neq\emptyset$. This can be done with various sets of parameters, but excitatory interactions plays an important role to reset the potential of all the neuron excepted one. \em
\end{Remark}

%any other regions of the phase space $\Gamma'_i$ such that $V_i\in (c^{\ast},d)$ with $d<\theta$ and isomorphic to an interval of $\R$ and any other set of parameters ensuring $\rho(\Gamma'_i)\cap\Gamma'_j\neq\emptyset$ and symmetrically that $\rho(\Gamma_i)\cap\Gamma_j\neq\emptyset$ for $i\neq j$ will allow to prove Theorem \ref{SNOCONT}, without changing its proof. In particular condition (O2) can be replace by the more biologically realistic hypothesis.   $H_{sk}>\theta-\delta$.

%Theorem \ref{SNOCONT} is proved using the regions $\Gamma_i$ of the phase space and a particular set of parameters defined by the condition (O1) to (O3). But any other regions of the phase space $\Gamma'_i$ such that $V_i\in (c^{\ast},d)$ with $d<\theta$ and isomorphic to an interval of $\R$ and any other set of parameters ensuring $\rho(\Gamma'_i)\cap\Gamma'_j\neq\emptyset$ and symmetrically that $\rho(\Gamma_i)\cap\Gamma_j\neq\emptyset$ for $i\neq j$ will allow to prove Theorem \ref{SNOCONT}, without changing its proof. In particular condition (O2) can be replace by the more biologically realistic hypothesis.   $H_{sk}>\theta-\delta$.

\begin{Remark}\em
1) The result of Theorem \ref{SNOCONT} may seem unexpected, since the non zero components of the return map are just translations applied to the time $t$ map $\phi^t$, which has a uniform negative exponent Lyapunov $- \gamma$. The return map fails to be piecewise contractive because
the waiting time ${\bar t}(\cdot)$ to enter in the firing regime depends on initial state. Therefore, when  the values of the return map at two different points are compared, we evaluate $\phi^t$ with different values of $t$, by mean of the map $\phi^{{\bar t}(\cdot)}(\cdot)$, which is not necessarily contractive. In particular, $\phi^{{\bar t}(\cdot)}(\cdot)$ is expansive in any set $\Gamma_i$, as shown by (\ref{ISIDIF}). This property is transmitted to the return map in any intersection of the sets $\Gamma_i$ with an atom of the partition $\mathcal{P}$, independently from the values of the interactions.

\vspace{1ex}
\noindent 2) The existence of expansion in a \em non trivial \em and \em invariant \em compact part of the phase space is considered as a source of chaos. In general the sets $\Gamma_i$ are not invariant. The role of the conditions (O1) to (O3), is precisely to guaranty the existence of an invariant subset in $\Gamma_i\cup\Gamma_j$. Indeed, for the networks verifying these conditions, it is possible to show that the invariant subset $\bigcap_{m \geq 0} \rho^{-m}(\Gamma_{i} \cup \Gamma_{j})$ is non-empty\footnote{The second part of the proof of Lemma \ref{PARAM} is the first step of the proof of  this statement.}. Unfortunately, it fails to be a chaotic attractor because it is trivial: it is just a periodic repeller (period 2). Nevertheless, it is a sufficient result to prove that the return map
cannot be piecewise contractive in the whole phase space for any value of the parameters, even equipping $\Sigma$ with a different norm. On the other hand, we cannot exclude the possibility of the existence of more complicated regions of the phase space and of the parameters for which there exists a non trivial chaotic set and we hope the proof of Theorem \ref{SNOCONT} gives a strategy to search it.
%: when searching for chaotic dynamics, look at those orbits such that the inter spike intervals ${\bar t}(\cdot)$ are more sensitive to the initial state.
\em
\end{Remark}

Lemma \ref{LEMDILAT2} establishes the existence of some regions of the Poincar\'e section where the return map expands the distances, the sets $\Gamma_i$. However, there is always a subregion of $\Sigma$ which is a contractive zone, as stated by the following Proposition.

%Lemma \ref{LEMDILAT2} establishes that there exists some regions of the Poincar\'e section where the return map expands the distances, the sets $\Gamma_i$. However, when $\beta<\frac{4}{3}\theta$, there is also a region of $\Sigma$ which is a contractive zone. This subset, denoted $\mathcal{C}_c$, consists of the initial states of the network such that the potential of all the  neurons is strictly smaller than the quantity
%\begin{equation}\label{C}
%c:=\beta-\frac{1}{2}(\beta-\theta+\sqrt{(\beta-\theta)^2 +4(\beta-\theta)(\beta-\alpha)}).
%\end{equation}
%Note that if $\beta\geq\frac{4}{3}\theta$ then $c\leq 0$ and $\mathcal{C}_c\cap\Sigma$ is empty. Biologically, the condition  $\beta<\frac{4}{3}\theta$  means that the external current $I_{ext}$ applied to each neuron induces at most an increase of potential of the same order of magnitude as the threshold potential (recall that $\beta=K/\gamma=RI_{ext}$). As a strong current $I_{ext}$ is a strong excitation, it is not surprising that a contractive zone may fail to exist if $\beta$ is much bigger that $\theta$.

\begin{Proposition}\label{CONTRACTIVEZONE}  For any $c\in[0,\theta]$, let us define the subset $\mathcal{C}_c$ of the Poincar\'e section  by
\begin{equation}\label{FORMCONTRACTIVEZONE}
\mathcal{C}_c:=\{\V\in \Sigma\ :\ \alpha\leq V_i\leq c\ \forall\, i\in I\}.
\end{equation}
The set $\mathcal{C}_0$ is a contractive zone with respect to the partition $\mathcal{P}$. Moreover, if
\[
\beta<\beta_{+}(\alpha) \quad \text{with} \quad \beta_{+}(\alpha):=\frac{1}{2}(\alpha+2\theta+\sqrt{\alpha^2+4\theta^2})
\]
then for any $c\in(0,{\bar c})$ where
\begin{equation}
{\bar c}:=\beta-\frac{1}{2}\left(\beta-\theta+\sqrt{(\beta-\theta)^2 +4(\beta-\theta)(\beta-\alpha)}\right),
\end{equation}
the set $\mathcal{C}_c$ is non empty, and is a contractive zone with respect to the partition $\mathcal{P}$.
\end{Proposition}
For any value of the parameters, there is a contractive zone in the Poincar\'e section: the set $\mathcal{C}_0$. Moreover,  if some restrictions are imposed to $\beta$, namely $\theta<\beta<\beta_{+}(\alpha)$, then ${\bar c}>0$, and there exist larger contractive zones: all the sets $\mathcal{C}_c$ with $0<c<{\bar c}$. The nearer $\beta$ is to $\theta$, the larger is ${\bar c}$. In particular, when $\beta$ tends to $\theta$, the whole Poincar\'e section tends to be a contractive zone. On the other hand when $\beta$ tends to $\beta_{+}(\alpha)$ the quantity ${\bar c}$ tends to $0$.

The quantity $\beta_{+}(\alpha)$ is an increasing function of $\alpha$ and therefore
\[
\lim\limits_{\alpha\to-\infty}\beta_{+}(\alpha)<\beta_{+}(\alpha)<\beta_{+}(0),
\quad
\text{that is}\quad \beta_{+}(\alpha)\in(\theta,2\theta)\quad \forall\,\alpha<0.
\]
It follows that there always exists some values of $\beta$ satisfying the condition $\theta<\beta<\beta_{+}(\alpha)$, even for $\alpha$ arbitrary small. In particular if $\alpha$ is near to $0$ then
$\beta$ can be chosen near to $2\theta$. However, if $\beta$ is larger than $2\theta$, then the only set $\mathcal{C}_{c}$ which is a contractive zone is $\mathcal{C}_{0}$.

Recall that $\beta$ is proportional to the external current $I_{ext}$ applied to each neuron.
Therefore, the previous analysis reflects the fact that weak currents are more favorable to the presence of contraction (or dissipation) in the model.

\begin{proof}{\em of Proposition \ref{CONTRACTIVEZONE} \em} First, note that $\mathcal{C}_c\neq\emptyset$ if and only if $c\geq 0$, since any point of $\Sigma$ has at least a component equal to $0$. The set $\mathcal{C}_0$ is always non empty, but for a set $\mathcal{C}_c$ with $0<c<{\bar c}$
to be non empty we have to ensure that ${\bar c}>0$. This last condition is satisfied if and only if
\[
2\beta-(\beta-\theta)>\sqrt{(\beta-\theta)^2+4(\beta-\theta)(\beta-\alpha)}.
\]
Since both sides of the inequality are positive, elevating to the square each sides we obtain an equivalent expression, which after simplifications is:
\[
\beta^2-(\alpha+2\theta)\beta+\alpha\theta<0.
\]
This equation is satisfied if and only if $\beta$ verifies:
\[
\beta_{-}(\alpha):=\frac{\alpha+2\theta-\sqrt{\alpha^2+4\theta^2}}{2}<\beta<\frac{\alpha+2\theta+\sqrt{\alpha^2+4\theta^2}}{2}=\beta_{+}(\alpha).
\]
%For $\alpha\leq 0$, both $\beta_{-}(\cdot)$ and $\beta_{+}(\cdot)$ are increasing function of $\alpha$.
As $\beta_{-}(\cdot)$ is an increasing function of $\alpha$, it follows that $\beta_{-}(\alpha)< \beta_{-}(0)=0$, which is always smaller than $\beta$ (recall that $\beta>\theta>0>\alpha$). We conclude that for all $c\in(0,{\bar c})$ the set $\mathcal{C}_c\neq\emptyset$ if $\beta<\beta_{+}(\alpha)$.

\vspace{1ex}
Now we show that for all $c\in[0,\bar{c})$ there exists $\lambda_{c}<1$ such that for any $J\in {P}(I)$ we have:
\begin{equation}\label{LIP}
\|\rho({\bf V})-\rho({\bf W})\|\leq\lambda_{c}\|{\bf V}-{\bf W}\| \qquad  \forall\ {\bf V},{\bf W}\in
\Sigma_J\cap \mathcal{C}_c.
\end{equation}
Suppose ${\bf V},{\bf W}\in \Sigma_J\cap \mathcal{C}_c$ and let $i,l\in J$ be such that ${\bf V}\in \Sigma_i$ and ${\bf
W}\in \Sigma_l$ (recall (\ref{PARTITIONBI})). Let $k\in I$. If $k\in J$ then by definition of $\rho$
\begin{equation}\label{KINJ}
|\rho_k({\bf V})-\rho_k({\bf W})|=0.
\end{equation}
If $k\notin J$ we have to consider $4$ cases:

\vspace{1ex}
\noindent{\bf Case 1:} If $\rho_k(\V)=\rho_k(\W)=\alpha$ then (\ref{KINJ}) is true.

\vspace{1ex}
\noindent{\bf Case 2:} If $\rho_k(\V)>\alpha$ and $\rho_k(\W)>\alpha$ then, as $\V$ and $\W$ belong to the same set $\Sigma_J$ we have
\[
|\rho_k({\bf V})-\rho_k({\bf W})|  = |\phi_k^{\bar{t}(\V)}(\V)  - \phi_k^{\bar{t}(\W)}(\W)|
\]
Now, $\V\in\Sigma_i$ and $\W\in\Sigma_l$ imply $\overline{t}({\bf V})=t_i({\bf V})$ and $\overline{t}({\bf W})=t_l({\bf W})$. Using (\ref{PHITV}) with we obtain:
\begin{eqnarray*}
|\rho_k({\bf V})-\rho_k({\bf W})| & = &
\left|\frac{(\beta-W_k)(\beta-\theta)}{\beta-W_l}-\frac{(\beta-V_k)(\beta-\theta)}{\beta-V_i}\right|\\
     & = & \left|\frac{\beta-\theta}{\beta-W_l}(V_k - W_k)
     +\frac{(\beta-V_k)(\beta-\theta)}{(\beta-W_l)(\beta-V_i)}(W_l-V_i)\right|
\end{eqnarray*}
If $\V$ and $\W$ belongs to $\mathcal{C}_0$ then necessarily $W_l=V_i=0$. If not, as $W_l$ and $V_i$ are the largest components of $\W$ and $\V$ respectively, all the components of these points
are negative, which is impossible since they belong to $\Sigma.$ Therefore,
\begin{equation}\label{CONTRACTRATEC0}
|\rho_k({\bf V})-\rho_k({\bf W})|  =  \frac{\beta-\theta}{\beta}|V_k - W_k|\leq  \frac{\beta-\theta}{\beta}\|{\bf V}-{\bf W}\| \quad\text{if}\quad \V,\W\in\mathcal{C}_0.
\end{equation}
If $\V$ and $\W$ belongs to $\mathcal{C}_c$ for a $c\in(0,\bar{c})$ then
\[
|\rho_k({\bf V})-\rho_k({\bf W})| \leq \frac{\beta-\theta}{\beta-c}|V_k - W_k|
     +\frac{(\beta-\alpha)(\beta-\theta)}{(\beta-c)^2}|V_i-W_l|.
\]
Suppose $V_i\leq W_l$. As $t_i({\bf V})=\overline{t}({\bf V})$, we have $V_l\leq V_i$, which implies
$|V_i-W_l|\leq |V_l-W_l|$, and
\begin{eqnarray}\label{CONTRACTRATE}
|\rho_k({\bf V})-\rho_k({\bf W})|
     & \leq &\frac{\beta-\theta}{\beta-c}|V_k - W_k|
     +\frac{(\beta-\alpha)(\beta-\theta)}{(\beta-c)^2}|V_l-W_l| \nonumber \\
     & \leq &\frac{\beta-\theta}{\beta-c}\left(1
     +\frac{\beta-\alpha}{\beta-c}\right)\|{\bf V}-{\bf W}\|
     \quad\text{if}\quad \V,\W\in\mathcal{C}_c.%,\, c\in(0,\bar{c}).
\end{eqnarray}
Suppose $W_l< V_i$. As $t_l({\bf W})=\overline{t}(\W)$, we have $W_l\geq W_i$, which implies
$|V_i-W_l|\leq |V_i-W_i|$, and the inequality ($\ref{CONTRACTRATE}$) follows as well.

\vspace{1ex}
\noindent{\bf Case 3:} If $\rho_k(\V)>\alpha$ and $\rho_k(\W)=\alpha$ then
\begin{eqnarray*}
|\rho_k({\bf V})-\rho_k({\bf W})|
&= &   \phi_k^{\bar{t}(\V)}(\V)+\sum_{j\in J}H_{jk}-\alpha\\
&\leq& \phi_k^{\bar{t}(\V)}(\V) + \sum_{j\in J}H_{jk} - \phi_k^{\bar{t}(\W)}(\W)-\sum_{j\in J}H_{jk}\\
&\leq& \phi_k^{\bar{t}(\V)}(\V)  - \phi_k^{\bar{t}(\W)}(\W)
\end{eqnarray*}
and we obtain (\ref{CONTRACTRATEC0}) and (\ref{CONTRACTRATE}) by the same calculation as in Case $2$.

\vspace{1ex}
\noindent{\bf Case 4:} If $\rho_k(\V)=\alpha$ and $\rho_k(\W)>\alpha$ then, substituting $\V$ for $\W$ and $\W$ for $\V$ in Case 3, we obtain (\ref{CONTRACTRATEC0}) and (\ref{CONTRACTRATE}).

\vspace{1ex}
To sum up, if $\V,\W\in\mathcal{C}_0$, for all $k\in I$ either (\ref{KINJ}) or (\ref{CONTRACTRATEC0}) is true and
\[
\|\rho({\bf V})-\rho({\bf W})\| \leq \lambda_0\|{\bf V}-{\bf
W}\| \quad\mbox{where}\quad
\lambda_0=\frac{\beta-\theta}{\beta}<1.
\]
If $\V,\W\in\mathcal{C}_c$, for all $k\in I$ either (\ref{KINJ}) or (\ref{CONTRACTRATE}) is true and
\[
\|\rho({\bf V})-\rho({\bf W})\| \leq \lambda_c\|{\bf V}-{\bf
W}\| \quad\mbox{where}\quad
\lambda_c=\frac{\beta-\theta}{\beta-c} \left(1
+\frac{\beta-\alpha}{\beta-c}\right).
\]
Using the definition of ${\bar c}$ we obtain $\lambda_{\bar c}=1$. The quantity $\lambda_{c}$ being an increasing function of with $c$, we have $\lambda_c<\lambda_{\bar c}$ for all $c<{\bar c}$, which ends the proof.
\qed
\end{proof}

In this section we have shown that there exists some regions of the phase space where the return map expands the distances (the sets $\Gamma_i$) and other regions where it contracts them (the sets $\mathcal{C}_c$). Moreover, for some values of the interactions, it is not possible to make the return map piecewise contractive in the sets $\Gamma_i$ by changing the working  norm. This proves that the piecewise contractive character of the return map is not a trivial property, and to ensure it in the whole phase space, it is necessary to impose adequate conditions to the parameters of the model. In the forthcoming section we propose such conditions.

\subsection{Parameters of global contraction}  \label{SubsectionH3H4}

%In the previous subsection we have shown that $\Sigma$ is  not  {necessarily} a contractive zone with respect to $
%\mathcal{P}$. Nevertheless, we have shown in Proposition \ref{CONTRACTIVEZONE} that under the hypothesis $\beta<
%\frac{4}{3}\theta$ there is a subset $\mathcal{C}_c$ of $\Sigma$ which is a contractive zone.
%Here, we give values of the parameters such that this contractive zone is forward invariant.
%Moreover, if the network contains inhibitory neurons,  {we will prove that} any orbit drops into $\mathcal{C}_c$ after a uniform number of iterations. It allows us to prove that there is a metric making $\Sigma$ a contractive zone, with respect to a refined partition of $\mathcal{P}$.

Proposition \ref{CONTRACTIVEZONE} gives concrete examples of a regions of the Poincar\'{e} section, denoted $\mathcal{C}_c$ with $c<{\bar c}$, where the return map is contractive. The set $\mathcal{C}_{\bar c}$ is not a contractive zone in the sense of Definition \ref{CONTRACTIVEZONEDEF} because for this set $\lambda_{\bar c}=1$. Nevertheless, it contains all the contractive  zones $\mathcal{C}_c$ with $c<{\bar c}$ and will help the proof of the results of this section. In the sequel we give conditions on the parameter values (Hypothesis (H3) and (H4)) for which we prove that $\mathcal{C}_{\bar c}$ has two important features. First, once an orbit enters into $\mathcal{C}_{\bar c}$, it does not leave it, furthermore, it stays in a contractive zone contained in $\mathcal{C}_{\bar c}$ (Proposition \ref{INVARIANT}). Second, if at least one neuron is inhibitory, then any orbit will finally enter into $\mathcal{C}_{\bar c}$ (Proposition \ref{ATTRACT}). These two important properties of $\mathcal{C}_{\bar c}$ allow to show that after changing adequately the metric in $\Sigma$ and considering a refined partition of $\mathcal{P}$, the return map becomes piecewise contractive in the whole Poincar\'{e} section (Theorem \ref{TMETRIC}). This result holds under hypothesis (H1) to (H4) and if at least one of the neurons is inhibitory. The case for which all the neurons are excitatory, which mainly provokes the synchronization of the whole network, will be studied in Section \ref{SYNCHRO}, without assuming Hypothesis (H1) to (H4).

\vspace{1ex}
Along this section, we suppose that the parameters verify the following two hypothesis:

\vspace{1ex}
\noindent {\bf (H3)} The external current $I_{ext}$ is not too strong, in such a way that $\beta<\beta_{+}(\alpha)$ and the interactions $H_{ji}$ satisfy:
\[
\min_{j\neq i}{|H_{ji}|}> \epsilon \quad\text{where}\quad \epsilon := \frac{1}{2}\left(\sqrt{(\beta-\theta)^2 +
4(\beta - \theta)(\beta-\alpha)}-(\beta-\theta)\right).
\]

\vspace{1ex}
\noindent {\bf (H4)} We assume the Dale's principle: a neuron is either excitatory or inhibitory. In other words, the network does not contain mixed neurons.

\vspace{1ex}
\noindent Hypothesis (H3) and (H4) define the region of the parameters for which we study the global dynamics of
IF neural networks later on, unless otherwise specified. Hypothesis (H3) is an open condition relating the strength of the external current and of the resistance of the membrane (recall that $\beta=RI_{ext}$) with the intensity of the interactions. It first guaranties the existence of a contractive zone larger than $\mathcal{C}_0$ by imposing $\beta<\beta_{+}(\alpha)$ (see Proposition \ref{CONTRACTIVEZONE}). On the other hand, it imposes to the interactions to have a lower bound $\epsilon>0$ (but no upper bound is imposed). This lower bound is an increasing function of $\beta-\theta$. In the extreme case, where $\beta-\theta$ tends to $\beta_{+}(\alpha)-\theta$, the lower bound tends to $\theta$. But, if $\beta-\theta$ tends to $0$ then, to satisfy (H3), the interactions have to be different from $0$, but can be  arbitrarily near to $0$. In other words, given fixed values of the interactions, our description of the dynamics of the network  holds for a range of external currents such that $\beta=RI_{ext}$ is sufficiently near to $\theta$. Or equivalently, given a fixed value of $RI_{ext}$, this description apply to all the networks with sufficiently strong interactions.

\vspace{1ex}
In Section \ref{RETURNMAP} we noted that, for every $\V\in \Sigma$ such that $J_0({\bf V})$ contains only inhibitory neurons, $J(\V)=J_0({\bf V})$. If moreover $\V\in \mathcal{C}_{\bar c}$, the set $J_0(\V)$ has  the
following additional property:
\begin{Lemma}\label{JVAC}  If ${\bf V} \in \mathcal{C}_{\bar c}$ and an excitatory neuron fires spontaneously, that is $J_0({\bf V})$ contains an excitatory neuron, then all the neurons fire together, that is $J({\bf V})=I$.
\end{Lemma}
%\begin{Lemma}\label{JVAC}  If ${\bf V} \in \mathcal{C}_{\bar c}$ and $J_0({\bf V})$ contains an
%excitatory neuron, then $J({\bf V})=I$.
%\end{Lemma}

\begin{proof} Suppose ${\bf V} \in  \mathcal{C}_{\bar c}$ and let $i\in J_0(\V) \neq I$.
Let us compute the neurons of $J_1(\V)$. Suppose $k\in I\setminus J_0(\V)$. As ${\bar t}(\V)=t_i({\V})$, by (\ref{PHITV}) we have:
\[
\phi_k^{\overline{t}({\bf V})}({\bf V}) =\beta-\frac{(\beta-V_k)(\beta-\theta)}{\beta-V_i}.
\]
As $\V\in\mathcal{C}_{\bar c}$, we have $\alpha\leq V_k$ and $V_i\leq {\bar c}$ which implies
\begin{align*}
\phi_k^{\overline{t}({\bf V})}({\bf V}) & \geq\beta-\frac{(\beta-\alpha)(\beta-\theta)}{\beta-{\bar c}}\\
 & \geq \beta-\frac{4(\beta-\alpha)(\beta-\theta)}{2(\beta-\theta +\sqrt{(\beta-\theta)^2 + 4(\beta-\theta)(\beta-\alpha)})}\\
 & \geq \beta- \frac{1}{2}\sqrt{(\beta-\theta)^2 + 4(\beta-\theta)(\beta-\alpha)}=\frac{1}{2}(\beta+\theta)-\epsilon>\theta -\epsilon.
\end{align*}
Since $J_0({\bf V})$ contains an excitatory neuron, it follows that
\[
\phi_k^{\overline{t}({\bf V})}({\bf V})+\sum\limits_{j\in J_0({\bf V})\ :\ H_{jk}>0}
H_{jk}>\theta -\epsilon +\min_{j\neq i}{|H_{ji}|}>\theta,
\]
and by definition of $J_1(\V)$, we have that $k\in J_1(\V)$. We deduce that for all $k\in I\setminus J_0(\V)$ we have $k\in J_1(\V)$. Thus, for all $k\in I$ either $k\in J_0(\V)$ or $k\in J_1(\V)$. In both case $k\in J(\V)$.
\qed
\end{proof}

The Lemma states that for the initial states $\V\in\mathcal{C}_{\bar c}$, it is enough that a single excitatory neuron fires to make the whole network firing simultaneously and to provoke its synchronization (for more details
about synchronization, see Section \ref{SYNCHRO}). The reason is that, if $\V\in\mathcal{C}_{\bar c}\cap\Sigma_i$, the potential of the neuron $i$ is small enough for the other neurons getting sufficiently near to the threshold just before $i$ reaches it. In other words, all the neurons get spontaneously near to the threshold in a similar time, i.e. $t_j(\V)\simeq{\bar t}(\V)$ for all $j\in I$ if $\V\in\mathcal{C}_{\bar c}$. Therefore, even a small excitation (but bigger than $\epsilon$) is enough to help all the other neurons to reach the threshold in the avalanche process. Nevertheless, for initial states $\V \not \in \mathcal{C}_{\bar c}$, even if a single neuron may produce the avalanche, it does not necessarily makes all the neurons fire simultaneously.

%\noindent This lemma gives a first result about the dynamics of the network. If $\V\in \mathcal{C}_{\bar c}\cap \Sigma_i$, where $i$ is an excitatory neuron, then $\rho(\V)=0$. The origin of the phase space being a fixed point of $\rho$, it is a periodic orbit of the model. Moreover, this orbit corresponds to a synchronized state of the network, i.e the potential of all the neurons is the same at each time

\vspace{1ex}
Now we apply Lemma \ref{JVAC} to prove the existence of a forward invariant contractive zone in the the Poincar\'e section. We recall, that in such a case $\rho$ is said to be piecewise contractive.

\begin{Proposition}\label{INVARIANT} The set $\Sigma^*:=\overline{\rho(\mathcal{C}_{\bar c})}$ is contained in
$\mathcal{C}_{\bar c}$, is a contractive zone, and is such that $\rho(\Sigma^*)\subset\Sigma^*$.
\end{Proposition}

\begin{proof} We show that $\rho(\mathcal{C}_{\bar c})$ is contained in one of the contractive zones $\mathcal{C}_c$  of Proposition \ref{CONTRACTIVEZONE}. Since the sets $\mathcal{C}_c$ are compact, the same will be true for the closure $\Sigma^*:=\overline{\rho(\mathcal{C}_{\bar c})}$. Moreover, as $\mathcal{C}_{\bar c}$ contains all the sets $\mathcal{C}_c$, it will follow that $\Sigma^*\subset \mathcal{C}_{\bar c}$ and
$\rho(\Sigma^*)\subset\rho(\mathcal{C}_{\bar c})\subset\overline{\rho(\mathcal{C}_{\bar c})}=\Sigma^*$.

Let ${\bf V}\in \mathcal{C}_{\bar c}$ and $k\in I$. Since by (\ref{RETURN}) we know that
$\alpha\leq\rho_k({\bf V})$, we have only have to prove that $\rho_k({\bf V})\leq c$ for some $c\in[0,{\bar c})$.
 Suppose $k\notin J({\bf V})$, then $J(\V)\neq I$, and as $\V\in\mathcal{C}_{\bar c}$, from Lemma \ref{JVAC} we deduce that $J_0({\bf V})$ contains only inhibitory neurons. It implies that $\V\in\Sigma_J$ where $J:=J({\bf V})=J_0({\bf V})$. Since $k$ receives only inhibitory interactions from the neurons of $J$, we have:
\[
\rho_k(\V)=\phi_k^{\overline{t}({\bf V})}({\bf V})+\sum\limits_{j\in J}H_{jk}<\theta-\min_{j\neq i}{|H_{ji}|}<\theta-\epsilon.
\]
A straightforward computation shows that $\theta-\epsilon={\bar c}$ and we have that
$\rho_k(\V)\leq\theta-\min\limits_{j\neq i}{|H_{ji}|}<{\bar c}$. Now, if  $k\in J({\bf V})$ then $\rho_k({\bf V})=0<{\bar c}$. It follows that $\rho_k(\V)\leq\max\{0,\theta-\min\limits_{j\neq i}{|H_{ji}|}\}<{\bar c}$ for all $k\in I$. We conclude that $\rho(\V)$ belongs to the contractive zone $\mathcal{C}_c$ with $c=\max\{0,\theta-\min\limits_{j\neq i}{|H_{ji}|}\}$. \qed
\end{proof}

Proposition \ref{INVARIANT} ensures that if an orbit of the return map falls in the contractive zone $\Sigma^*$, then it stays forever in this region. The following proposition states that if the network contains an inhibitory neuron, then all the orbits eventually reach the contractive zone.

\begin{Proposition}\label{ATTRACT} If the network contains an inhibitory neuron and if $p\in\N$ is such that
$\alpha+p\min_{j\neq i}{|H_{ji}|}\geq\theta$, then $\rho^{p+1}(\Sigma)\subset\Sigma^*$.
\end{Proposition}

\begin{proof} Let us define for each $p\in\N$ the set $Z_p$ of the points of the
Poincar\'e section that stay outside of $\mathcal{C}_{\bar c}$ during $p$ iterations:
\[
Z_p:=\{\V\in \Sigma\ :\ \rho^j(\V)\in \Sigma\setminus \mathcal{C}_{\bar c},\ 1\leq j\leq p\}.
\]
We are going to show that if the network contains an inhibitory neuron and $p$ is such that $\alpha+p\min_{j\neq i}{|H_{ji}|}\geq\theta$, then $Z_p$ is empty, that is $\rho^p(\Sigma)\subset\mathcal{C}_{\bar c}$, and by Proposition \ref{INVARIANT}, $\rho^{p+1}(\Sigma)\subset\Sigma^*$. Afterwards, $i$ denotes an inhibitory neuron of the network.

\vspace{1ex}
First suppose that $Z_1\neq\emptyset$ and take $\V\in Z_1$. By construction of the set $Z_1$, if $\V\in Z_1$, then $\rho_k(\V)>{\bar c}$ for some $k\in I$. This neuron $k$ cannot fire at time ${\bar t}(\V)$, since if it would then $\rho_k(\V)=0$, which contradicts $\rho_k(\V)>{\bar c}$. Therefore,
\[
\phi_k^{{\bar t}(\V)}(\V)+\sum_{j\in J(\V)\ : \ H_{jk}>0} H_{jk} <\theta
\]
and it follows that
\[
\rho_k(\V)=\phi_k^{{\bar t}(\V)}(\V)+\sum_{j\in J(\V)\ : \ H_{jk}>0} H_{jk} +\sum_{j\in J(\V)\ : \
H_{jk}<0}H_{jk}<\theta  +\sum_{j\in J(\V)\ : \ H_{jk}<0} H_{jk}.
\]
If $J(\V)$ contains an inhibitory neuron then $\rho_k(\V)<\theta-\min_{j\neq i}{|H_{ji}|}={\bar c}$, which is a contradiction.
Thus $J(\V)$ contains only excitatory neurons and $i$ does not belong to $J(\V)$. Therefore, the component $i$ of $\rho(\V)$ satisfies:
\[
\rho_i(\V)=\phi_i^{{\bar t}(\V)}(\V)+\sum_{j\in J(\V)} H_{ji}=(V_i-\beta)e^{-\gamma{\bar t}(\V)}+\beta+\sum_{j\in
J(\V)} H_{ji}>V_i+\min_{j\neq i}{|H_{ji}|}.
\]
So, we have proved that if $\V\in Z_1$ then $\rho_i(\V)> V_i+\min_{j\neq i}{|H_{ji}|}$.

\vspace{1ex}
If we suppose now that $Z_p\neq\emptyset$ for a $p>1$ and if we take $\V\in Z_p$, then $\rho^j(\V)\in Z_1$ for all $j\in\{0,\dots,p-1\}$. By induction it follows that
\begin{align*}
\theta>\rho_i^p(\V)=\rho_i(\rho^{p-1}(\V))&>\rho_i^{p-1}(\V)+\min_{j\neq i}{|H_{ji}|}\\
%=\rho_i(\rho^{p-2}(\V))+\min_{j\neq i}{|H_{ji}|}
&>\rho_i^{p-2}(\V)+2\min_{j\neq i}{|H_{ji}|}>\dots>V_i+p\min_{j\neq i}{|H_{ji}|}
%\geq\alpha +p\min_{j\neq i}{|H_{ji}|}.
\end{align*}
%\[
%\theta>\rho_i^p(\V)=\rho_i(\rho^{p-1}(\V))>\rho_i^{p-1}(\V)+\min_{j\neq i}{|H_{ji}|}
%=\rho_i(\rho^{p-2}(\V))+\min_{j\neq i}{|H_{ji}|}
%>\rho_i^{p-2}(\V)+2\min_{j\neq i}{|H_{ji}|} \linebreak[4]{ >\dots>V_i+p\min_{j\neq i}{|H_{ji}|}\geq\alpha +p\min_{j\neq i}{|H_{ji}|}}.
%\]
As $V_i\geq \alpha$, we have just obtained that $\alpha+p\min_{j\neq i}{|H_{ji}|}<\theta$ if $Z_p\neq\emptyset$. Thus, if $\alpha+p\min_{j\neq i}{|H_{ji}|}\geq\theta$ the set $Z_p$ is empty and $\rho^p(\Sigma)\subset\mathcal{C}_{\bar c}$. \qed
%\qed
\end{proof}

\begin{Remark}\label{PEPS}\em Proposition \ref{ATTRACT} allows to obtain an estimation of the transient time $t_{trans}$
necessary for all the orbits of the network to enter into the contractive zone $\Sigma^*$. Let $p_0$ is the smallest integer such that $\alpha+p_0\min_{j\neq i}{|H_{ji}|}\geq\theta$, that is $p_0:=\left\lceil(\theta-\alpha)/\min_{j\neq i}{|H_{ji}|}\right\rceil$. Then any orbit of the network, with an initial condition in the Poincar\'e section $\Sigma$, enters into the contractive zone after $p_0$ returns in $\Sigma$. Now recall that the return time (or the inter-spike interval) is the necessary time for the network to enter in the firing regime from its initial state $\V\in\Sigma$, that is ${\bar t}(\V)$.
A straightforward computation using the equation (\ref{FLOW}) of the flow, shows that an upper bound of
${\bar t}(\V)$ for $\V\in\Sigma$ is $T := \log\left (\beta/(\beta - \theta)\right )$. Therefore, we deduce the following upper bound for $t_{trans}$:
\[
t_{trans}\leq p_0T=\left\lceil\frac{\theta-\alpha}{\min_{j\neq i}{|H_{ji}|}}\right\rceil\log\left(1+\frac{\theta}{\beta - \theta}\right).
\]
It follows that the stronger the interactions are, the shorter is the transient time to enter into the contractive zone. On the other hand,
%this upper bound is a decreasing function of $\beta-\theta$, and when it tends to $0$, the inter-spike intervals $T$ tends to infinity and the upper bound as well. Also,
if $\beta-\theta$ is small then the condition (H3) can permit very small  interactions. In such a case, the upper bound does not prevent from very long transients. Nevertheless,  the transient time remains always finite.
\em
\end{Remark}

\begin{Corollary}\label{EXIT} If there exists $\V\in\Sigma$ such that $\rho^p(\V)\notin\Sigma^*$ for all $p\in\N$
then the network contains only excitatory neurons.
\end{Corollary}

To close this section about the contraction properties of (the return map of) IF neural networks we state the following Theorem \ref{TMETRIC}. It proves that if the parameters of a network satisfy the hypothesis (H3) and (H4) then its return map is piecewise contractive in the whole Poincar\'e section.
%\vspace{.2cm}
%
%\noindent{\bf Theorem \ref{}.1} \em If the system satisfies Hypothesis (H1) and  (H3), if all the    neurons   are excitatory and if the initial state $\V$  is such that it eventually drops into the contractive zone $\Sigma $ defined in Proposition \em \ref{CONTRACTIVEZONE}, \em (for instance if $\V \in \Sigma$), then the system   synchronizes periodically, independently of the number of neurons in the network.
%
%\vspace{.1cm}
%
%{\em Proof: } It is not restrictive to assume that $\V \in \Sigma _c$. Since $I= I^+$ we have $J(\V) \subset I= I^+$. So, there exists an excitatory neuron $i \in J(\V)$. Applying Lemma \ref{} we obtain $J(\V) = I$, and this imply that $\rho(\V) = {\bf 0}$. $\ \ \Box$

\begin{Theorem}\label{TMETRIC} Under the hypothesis
{\em (H1), (H2), (H3) and (H4)}, if at least one neuron is inhibitory, then there exists an adapted metric
and a partition $\mathcal{P}'$ of the Poincar\'{e} section $\Sigma$ such that   the return map $\rho$ is piecewise contractive in $\Sigma$ with respect to $\mathcal{P}'$.
\end{Theorem}
The proof of the theorem is given appendix \ref{PROOFTMETRIC}. It follows from Proposition \ref{INVARIANT} and Proposition \ref{ATTRACT} which allow to prove that $\rho$ is ``eventually piecewise contractive" (see Lemma \ref{HYPOKATOK} in appendix \ref{PROOFTMETRIC}). This is a generalization for piecewise continuous maps of the definition of eventually contractive maps (see for instance Definition 2.6.11 in \cite{HK}). Then, the classical arguments to prove the existence of an adapted metric can be reproduced for piecewise continuous maps.

%\begin{proof} See appendix \ref{PROOFTMETRIC}.
%\qed
%\end{proof}

\vspace{1ex}
In the sequel of the paper we will based our study on Proposition \ref{INVARIANT} and Proposition \ref{ATTRACT}. They show that for the region of parameters defined by (H3) and (H4), the Poincar\'e section contains two important regions: the contractive zone ($\Sigma^*$), which is a forward invariant set of the return map, and the set of the points whose orbit never fall in the contractive zone. This last set can be non-empty only for networks composed exclusively of excitatory neurons  (cf Corollary \ref{EXIT}). We study these networks in the Section \ref{SYNCHRO}.
In the case of the networks containing inhibitory neurons, all the orbits drop into the contractive zone after a finite time (Proposition \ref{ATTRACT}). The study of the asymptotic dynamics reduces then to the analysis of the dynamics of the return map in the contractive zone. This is the purpose of Section \ref{ASYMPTOTIC}.

%%%%%%%%%%%%%%%%%%%%%%%%%%%%%%%%%%%%%%%%%%%%%%%%%%%%%%%%%%%%%%%%%%%%%%%%%%%%%%%%%%%%%%%%%%%

\section{Synchronization}\label{SYNCHRO}

In this section we prove the part (1) of Theorem \ref{Theorem1} stating sufficient conditions for the
global synchronization of a network.

\begin{Definition} \em We say that an orbit $\{\V(t)\}_{t \in \R^{+}}$ is \em a synchronized orbit \em of
the network, if $V_i(t)=V_j(t)$ for all $i,j\in I$ and $t\in\R^{+}$. We say that the network   \em globally
synchronizes, \em if for \em any \em initial state $\V(0)$ there exists $t_0\in\R^{+}$ such that
$\{\W(t)=\V(t+t_0)\}_{t \in \R^{+}}$ is a synchronized orbit.
%We say that the network is \em globally periodic and synchronized \em if besides $\{\W(t)\}_{t \geq 0, t \in\R}$ is   periodic.
\end{Definition}
Up to a time translation, a network admits only one synchronized orbits and this orbit is periodic. Indeed, suppose a synchronized orbit intersects the Poincar\'e section at a point $\V\in\Sigma$. Since $\V\in\Sigma$, one of its component is equal to zero, and since $\V$ belongs to a synchronized orbit, all its components are equal. It follows that any synchronized orbit intersects $\Sigma$ at the same point, which is the origin {\bf 0} of $\R^n$. Moreover, as in the sub-threshold regime all the neurons obey to the same the differential equation (\ref{IF}), if $\V={\bf 0}$, then all the neurons reach the threshold spontaneously at the same time and their potential is reset to $0$ at the same time. In other word $\rho({\bf 0})={\bf 0}$ and therefore the orbits of {\bf 0} is a periodic orbit of the network.

Latter on, we will call  $\V\in\Sigma$ a state of eventual synchronization, if there exists $l\in\N$ such that $\rho^{l}(\V)={\bf 0}$. Note that a network globally synchronizes if and only if all the points of the Poincar\'e section are states of eventual synchronization.

\begin{Theorem}\label{GLOBALSYNCHRO} Under hypothesis \em (H1), \em if the network is exclusively composed of excitatory neurons, that is $\min\limits_{j \neq i} H_{ji} >0$, and if the number $n$ of neurons satisfies
\begin{equation}\label{NUMBERNEURON}
n\geq\left\lceil\frac{\theta}{\min\limits_{j \neq i} H_{ji}}\right\rceil^2,
\end{equation}
where $\lceil x \rceil$ denotes the first integer larger or equal to $x$, then the network globally synchronizes. Moreover, the transitory time before the global synchronization
of the network is not larger than
%\[
%t_{trans}:=\frac{1}{\gamma}\left(\ln\left(\frac{\beta-\alpha}{\beta-\theta}\right)+
%\left\lceil\frac{\theta}{\min\limits_{j \neq i} H_{ji}}\right\rceil\ln\left(\frac{\beta}{\beta-\theta}\right)\right).
%\]
\[
t_{trans}:=\ln\left(\frac{\beta-\alpha}{\beta-\theta}\right)^\frac{1}{\gamma}+
\left\lceil\frac{\theta}{\min\limits_{j \neq i} H_{ji}}\right\rceil\ln\left(\frac{\beta}{\beta-\theta}\right)^\frac{1}{\gamma}.
\]
\end{Theorem}
In networks of excitatory neurons, the synchronization state is persistent. That is to say,
the network recovers its synchronization in a finite time, if a small perturbation is applied to the
orbit of the synchronized state ${\bf 0}\in\Sigma$. Precisely, the flow $\phi^t$ that solves the differential equation (\ref{IF}) depends continuously on the initial state $\bf V$ for  fixed $t>0$, provided that in the interval $(0,t)$ no neuron fires. Therefore, there exists $\delta^*>0$ such that if $\|\bf V - \bf 0\|< \delta^*$
then $\|\phi ^t({\bf V}) - \phi ^t ({\bf 0})\|< m$, where $m:=\min\limits_{j \neq i} H_{ji}$.
If $i$ is the first neuron to reach the threshold potential from the perturbed
state $\bf V$ at time ${\bar t(\V)}$, the potential of all the other neurons belongs to $(\theta- m, \theta]$ at this instant. Since the minimum interaction with $i$ is of magnitude $m$, the firing of $i$ induces the firing of the other neurons. The interactions being instantaneous,
the potential of $i$ and of the other neurons are reset to $0$ simultaneously and $\rho(\V)={\bf 0}$.
The same argument works as well to prove that the global synchronization is stochastically stable.
If a synchronized orbit suffers a stochastic perturbation such that, at the instant when a first neuron
reaches the threshold, the potential of the other neurons is sufficiently near the threshold, then
this orbit goes back to the Poincar\'e section at the synchronized state.

\begin{proof}{\em of Theorem \ref{GLOBALSYNCHRO}: \em} As the network does not contain inhibitory neurons,
a time $t_{+}$ from which the potential of all the neurons is larger or equal to zero exists. This time is smaller
than the necessary time for a neuron, with an initial potential equal to $\alpha$, to reach the threshold spontaneously. That is to say,
\[
t_{+}\leq \frac{1}{\gamma}\ln\left(\frac{\beta-\alpha}{\beta-\theta}\right),
\]
recall (\ref{FLOW}). Besides, after $t_+$, it is in the set
\[
\Sigma^{+}:=\{\V\in\Sigma\ :\ V_i\geq 0\ \forall\,i\in I\}
\]
that the orbits of the network intersect the Poincar\'e section. Therefore, to prove the global synchronization of the network, it is enough to do it for the initial states in $\Sigma^+$. Thus, we are going to show that there exists a natural number $l$ such that $\rho^{l}(\V)={\bf 0}$ for all $\V\in\Sigma^{+}$.

\vspace{1ex}
Let us denote
\begin{equation}\label{MANDP}
m: = \min_{j \neq i} H_{ji} >0 \quad\text{and}\quad p:=\left\lceil \frac{\theta}{m}\right\rceil,
\end{equation}
We state and prove two claims to achieve the proof of the theorem.

\vspace{1ex}
\noindent {\bf Claim 1:} For any initial state $\V$ of the network in $\Sigma^{+}$, each neuron fires
at least once before the $p^{\text{th}}$ return of the orbit of $\V$ in $\Sigma^{+}$. In brief, the set $J^p(\V)$ defined by
\[
J^p(\V):=\bigcup\limits_{j=0}^{p-1}J(\rho^j(\V))
\]
is equal to $I$ for any $\V\in\Sigma^{+}$.

\vspace{1ex}
\noindent Let $i\in I$. Let us show that if $i\notin J^{p-1}(\V)$, that is $i\notin J(\rho^{j}(\V))$ for all $j\in\{0,\dots,p-2\}$, then $i\in J(\rho^{p-1}(\V))$. If $i\notin J^{p-1}(\V)$, then, by definition of $\rho$, for any
$j\in\{1,\dots,p-1\}$ we have:
\[
\rho_i^j(\V)=\phi^{{\bar t}(\rho^{j-1}(\V))}_i(\rho^{j-1}(\V))+\sum_{k\in J(\rho^{j-1}(\V))} H_{ki}
> \rho_i^{j-1}(\V) + m,
\]
where the inequality is obtained using that $\phi^{{\bar t}(\V)}_i(\V)>V_i$ for all $\V\in\Sigma$ and $i\in I$,
and the fact that all the interactions are excitatory. By induction it follows that
$\rho_i^{p-1}(\V)>V_i+(p-1)m $.
We deduce that
\[
\phi^{{\bar t}(\rho^{p-1}(\V))}_i(\rho^{p-1}(\V))+\sum_{k\in J(\rho^{p-1}(\V))} H_{ki}
> \rho_i^{p-1}(\V) + m >V_i+pm \geq pm \geq \theta,
\]
which implies that $i\in J(\rho^{p-1}(\V))$ and proves the claim.

\vspace{1ex}
\noindent {\bf Claim 2:} For any initial state of the network $\V\in\Sigma^+$, there exists $j_0\in\{0,\dots,p-1\}$ such that the number of neurons that fire simultaneously at time ${\bar t}(\W)$, where $\W$ is the $j^{\text{th}}$ return of the orbit of $\V$ in $\Sigma^+$, is larger than $\theta/m$. In brief, for all $\V\in\Sigma^+$, there exists $j_0\in\{0,\dots,p-1\}$
such that  $\# J(\rho^{j_0}(\V))\geq\theta/m$.

\vspace{1ex}
\noindent
Let $\V\in\Sigma^+$ and $j_0$ be such that $\# J(\rho^{j_0}(\V))\geq \# J(\rho^{j}(\V))$ for all
$j\in\{0,\dots,p-1\}$. According to Claim 1, we have $n=\# J^p(\V)$. Therefore,
\[
n=\#J^p(\V)\leq\sum_{j=0}^{p-1}\# J(\rho^j(\V))\leq p\# J(\rho^{j_0}(\V)).
\]
Using (\ref{NUMBERNEURON}) we prove Claim 2:
\[
\# J(\rho^{j_0}(\V))\geq \frac{n}{p}=\left\lceil\frac{\theta}{m}\right\rceil.
\]

The quotient $\theta/m$ is the number of times that the minimum synaptic weight $m$ has to be added to
the potential of a neuron with the minimal potential $0$ to reach $\theta$. Claim 2 shows that at some finite instant, the number of neurons that fire simultaneously is at least
as large as this quotient. Thus, the firing of these neurons necessarily induces  the firing of any other, at the same instant.  Indeed, in the mathematical notation, using Claim 2 we obtain
\[
\phi^{{\bar t}(\rho^{j_0}(\V))}_i(\rho^{j_0}(\V))+\sum_{k\in J(\rho^{j_0}(\V))} H_{ki}
> \rho^{j_0}_i(\V) + m\frac{\theta}{m}\geq \theta,\quad \forall i\in I,\, \V\in\Sigma_+.
\]
Therefore, $J(\rho^{j_0}(\V))=I$ and $\rho^{l}(\V)={\bf 0}$ with $l=j_0+1$.

\vspace{1ex}
The global synchronization of the network need at most a time $t^+$ to enters in the
set $\Sigma^+$ and $j_0+1\leq p$ iterates of the return map. Thus, the transitory regime
does not exceed the time
\[
t_{trans}=\frac{1}{\gamma}\ln\left(\frac{\beta-\alpha}{\beta-\theta}\right) + \frac{p}{\gamma}\ln\left(\frac{\beta}{\beta-\theta}\right)=\frac{1}{\gamma}\left(\ln\left(\frac{\beta-\alpha}{\beta-\theta}\right)+
\left\lceil\frac{\theta}{m}\right\rceil\ln\left(\frac{\beta}{\beta-\theta}\right)\right).
\]
\qed
\end{proof}

Theorem \ref{GLOBALSYNCHRO} shows that a network of excitatory neurons globally synchronizes if the number of neurons $n$ is sufficiently large. This theorem only requires the hypothesis (H1), to ensures the obvious fact that   no neuron has an infinitely negative potential. As all the synapses are supposed excitatory, the hypothesis (H2) and (H4) are not needed. Likewise, the hypothesis (H3) of the existence of a lower bound $\epsilon>0$ for the interactions is not required. In particular, positive but arbitrarily small interactions may affect the length of the transitory regime, but do not prevent the network from globally synchronizing, provided $n$ is sufficiently large. Finally, if the number of neurons does not satisfies (\ref{NUMBERNEURON}) the global synchronization may not occur. We illustrate this statement with a simple example involving two neurons and constant interactions.

\begin{Proposition}\label{NOGLOBALSYNCHRO} Let $I=\{1,2\}$ and $H_{12}=H_{21}=H\in (0,\theta)$. The orbit $\{\rho^k(\V)\}_{k\in\N}$ of the point
\[
\V=(\beta-x,0)\quad\text{where}\quad x:=\frac{1}{2}\left(\sqrt{H^2+4\beta(\beta-\theta)}-H\right)
\]
is not synchronized. It is a period two orbit, the points of which are $(\beta-x,0)$ and $(0,\beta-x)$.
\end{Proposition}

\begin{proof} A straightforward computation shows that $\beta-x<\theta$ if and only if $H<\theta$, and $\beta-x>0$ if and only if $H>-\theta$. Thus, the condition $H\in(0,\theta)$ ensures that $\V\in\Sigma$ and not equal to $(0,0)$. Now, let us show that $\rho(\V)=(0,\beta-x)$. As $V_1>0=V_2$, the neuron $1$ fires spontaneously at time ${\bar t }(\V)=t_1(\V)$ and $J_0(\V)=\{1\}$. Using (\ref{PHITV}) with $V_2=0$, we obtain
\[
\phi_2^{{\bar t}(\V)}(\V) + H = \beta-\frac{\beta(\beta-\theta)}{(\beta-V_1)} + H=\beta-\frac{\beta(\beta-\theta)}{x}+ H = \beta -x.
\]
Since $\beta -x<\theta$, we deduce that the neuron $1$ does not induce the firing of the neuron $2$, and
$J_1(\V)=J_0(\V)$. Thus $J(\V)=\{1\}$, and
\[
\rho_1(\V)=0\qquad\text{and}\qquad\rho_2(\V)=\phi_2^{{\bar t}(\V)}(\V) + H =\beta -x,
\]
which proves that $\rho(\V)=(0,\beta-x)$. Substituting $V_1$ by $\rho_2(\V)$ and $V_2$ by $\rho_1(\V)$ in the
proof of $\rho(\V)=(0,\beta-x)$, we obtain $\rho^2(\V)=(\beta-x,0)=\V$ \qed
\end{proof}
The result of Proposition \ref{NOGLOBALSYNCHRO} can be easily extended to an arbitrary large number of neurons. For example, in a network with $2k=n$ neurons, with equal synaptic weight $H/k$ where $H\in(0,\theta)$, any initial state such that half of the neurons have their potential equal to $(\beta-x)$, and the other half have their potential equal to $0$, is a periodic point of period $2$.

\vspace{1ex}
When Hypothesis (H3) is assumed, and thus the interactions are bounded bellow by a positive number (which tends to zero when $(\beta-\theta)$ tends to zero), the global synchronization is not sure, since the hypothesis of Proposition \ref{NOGLOBALSYNCHRO} can be satisfied as well.
However, according to Lemma \ref{JVAC}, there is always some initial states for which the network synchornizes.

\begin{Proposition}\label{DEATHSYNCH.1} Under Hypothesis \em (H1) \em and \em (H3)\em, if the network is exclusively composed of excitatory neurons, and if its orbit visits the contractive zone $\Sigma^*$,  then the network synchronizes, independently of its size.

\end{Proposition}

\begin{proof} Suppose the orbit of the network visits $\Sigma^*$ at a point $\V$. Since  $\Sigma^*\subset\mathcal{C}_{\bar c}$ and $I$ is only composed of excitatory neurons, the hypothesis of Lemma \ref{JVAC} are satisfied. We deduce that $J(\V) = I$, and so $\rho(\V) = {\bf 0}$.
\qed
\end{proof}

Now, if the network has also inhibitory neurons, any orbit visits $\Sigma^*$ after a finite transitory regime, see Proposition \ref{ATTRACT}. Therefore, if the network does not synchronizes, it is because the excitatory neurons stop to fire after a finite time. We prove rigorously this results in what follows.

\begin{Definition} We say that $\V\in\Sigma$  is a state of eventual death of the neuron $i$, if
there exists $p\in\N$ such that $i\notin J(\rho^j(\V))$ for all $j\geq p$.
\end{Definition}
In other words, a state of eventual death of a neuron is a state such that the neuron stops to emit spikes after a
certain time. Therefore, we define a state of the (continuous time) model of eventual death as a state whose orbit
intersects the Poincar\'e section in a state of eventual death.

\begin{Proposition}\label{DEATHSYNCH} Under the hypothesis {\em (H1), (H2), (H3) and (H4)}, if the network is composed
of excitatory and inhibitory neurons, then the states of the network are either states of eventual synchronization
or states of eventual death of all the excitatory neurons.
\end{Proposition}

\begin{proof} Let $i$ be an excitatory neuron and $\V\in\Sigma$. By
Proposition \ref{ATTRACT} we can assume $\V\in\Sigma^*$ without loss of generality. We have, either $i\notin
J(\rho^j(\V))$ for all $j\in\N$ or there exists $j\in\N$ such that $i\in J(\rho^j(\V))$. In the first case $\V$ is
a state of eventual death of the neuron $i$. In the second case, by Proposition \ref{INVARIANT} it follows $\rho^j(\V)\in\Sigma^*$ for all $j\in\N$, and then we can apply Lemma \ref{JVAC} to deduce that $J(\rho^j(\V))=I$.
This implies $\rho^{j+1}(\V)={\bf 0}$. \qed
\end{proof}

%%%%%%%%%%%%%%%%%%%%%%%%%%%%%%%%%%%%%%%%%%%%%%%%%%%%%%%%%%%%%%%%%%%%%%%%%%%%%%%%%%%%%%%%%%%%%%%%%%%%%%%%

\section{Asymptotic dynamics of networks with inhibitory neurons}\label{ASYMPTOTIC}

In the previous section we studied the dynamics of networks composed exclusively of excitatory neurons. Now we focus on networks containing at least one inhibitory neuron. Along this section, we assume Hypothesis (H1), (H2), (H3) and (H4) on the parameter values, and that the set $I^-$ of inhibitory neurons is non empty.

Section \ref{subsectionPartitionSubc} is a preliminary section where a more precise characterization of the return map in the contractive zone $\Sigma^*$ is given. We study its continuity pieces and show that it satisfies a general definition of piecewise contracting maps. In Section \ref{STABLANDCAOTIC}, we introduce the concepts of stable set and sensitive set. Then, in Section \ref{PERIODSYST} we give and we comment the principal results about the asymptotic dynamics of the return map in both of these sets. It contains in particular Theorem \ref{teoremaPrincipal}, which is a more detailed version of the part 2) of Theorem \ref{Theorem1}, about the dynamics in the stable set. Finally, the ending Section \ref{PROOFCYCLESETA} is the proof of a proposition contained in Section \ref{PERIODSYST}.

\subsection{Continuity pieces of the return map in the contractive zone $\Sigma^*$.}\label{subsectionPartitionSubc}

In Proposition \ref{ATTRACT} we have shown that any orbit of a network containing inhibitory neurons finally drops into the forward invariant contractive zone $\Sigma^{\ast}$. Then, it is not restrictive to consider this space as our new phase space\footnote{From now on, $\rho$ will denote the return map restricted to $\Sigma_c$.}. In Section \ref{RETURNMAP}, we introduced the partition $\mathcal{P}$ of the Poincar\'e section defined by (\ref{PARTITION}). When restricting our the phase space to $\Sigma^*$, this partition becomes
\[
\mathcal{P}^*:=\{\Sigma^*_{J}\}_{J\in P(I)}\qquad\mbox{where}\qquad \Sigma^*_{J}:=\Sigma_J\cap \Sigma^*\quad\forall\, J\in P(I).
\]
As before, if the initial state of the network is $\V\in\Sigma^*_J$, then, when the network enters in the firing regime at time ${\bar t}(\V)$, the neurons that fire are exactly those of the set $J$.
As a starting point to the study of the asymptotic dynamics of neural networks containing inhibitory neurons, we give here a detailed description of the partition $\mathcal{P}^*$. It will allow to show that the return map satisfies a general definition of piecewise contractive map.
%In particular, we show that it is made of open pieces, where the return map is continuous, and of closed sets which are the boundary of the open pieces, where the return map generically presents a discontinuity jump.

\vspace{1ex}
First of all, some atoms of $\mathcal{P}^*$ are empty. Indeed, by Lemma \ref{JVAC}, if $J$ contains excitatory neurons and $J\neq I$, then $\Sigma^*_{J}=\emptyset$.
As a consequence, the possibly nonempty sets of $\mathcal{P}^*$ are: the set
$\Sigma^*_{I}$ of the initial states such that in the firing regime all the neurons fires, and the sets $\Sigma^*_{J}$ of the initial states such that in the firing regime only some inhibitory neurons fire. It follows that
\[
\Sigma^*=\bigcup_{J\in P\left(I^-\right)}\Sigma^*_{J}\bigcup\Sigma^*_{I},
\]
where $I^-\subset I$ is the set of the inhibitory neurons and $P(I^-)$ denotes the set of all the nonempty subsets of $I^-$.

\vspace{1ex}
Each set $\Sigma^*_{J}$ admits an explicit formulation which is useful to study the topological properties of $\mathcal{P}^*$.

Let $J\in P\left(I^-\right)$ and suppose the initial state of the network is $\V\in \Sigma^*_{J}$.
%By definition of $\Sigma^*_{J}$, in the firing regime the neurons that fire are those of $J$
%and are all inhibitory.
Since the firing of an inhibitory neuron cannot lead another neuron to fire,  if all the neurons of $J$ fire, it is because they all reach the threshold spontaneously at the same time ${\bar t}(\V)$.
In other words,  $t_i(\V)=\bar{t}(\V)$ for all $i\in J$ and $t_k(\V)>\bar{t}(\V)$ for all $k\notin J$. The initial potential of all the neurons of $J$ is thus
the same and it is larger than the potential of the other neurons. So, we can conclude that
\begin{equation}\label{SIGMAJINHIB}
\Sigma^*_{J} = \{\V \in \Sigma^*\ :\ V_i = V_j > V_k \quad
\forall\, i, j \in J \mbox{ and } \forall\, k \notin J\}\qquad
\forall\, J\in P(I^-).
\end{equation}

Now let $J=I$ and and suppose the initial state of the network is $\V\in\Sigma^*_{I}$. Then, either all the neurons reach spontaneously the threshold at the same time, or at least one excitatory neuron reaches the threshold before the others and induces the firing of all the others by avalanche. In both cases $t_i(\V)=\bar{t}(\V)$ for some $i$ in the set $I^{+}$ of the excitatory neurons. So, we can conclude that
\begin{equation}\label{SIGMAI}
\Sigma^*_{I} = \{\V \in \Sigma^*\ :\ \max_{i\in I^+} V_i \geq V_k
\quad \forall\, k \in I\}.
\end{equation}

\vspace{1ex}
Let us study the topological properties of the sets $\Sigma^*_{J}$. We consider the relative topology induced by $\R^{n}$ in $\Sigma$.

In the case where $J$ contains only one neuron, then it is an inhibitory neuron, and according to (\ref{SIGMAJINHIB})
\[
\Sigma^*_{\{i\}} = \{\V \in \Sigma^*\ :\ V_i > V_k \quad \forall\,
\, k \neq i\}\qquad \forall\, i\in I^-.
\]
A set $\Sigma^*_{\{i\}}$ is open, and its topological boundary, that we denote $\partial\Sigma^*_{\{i\}}$, is the set
\begin{equation}\label{BORDESIGMAJ}
\partial\Sigma^*_{\{i\}}=\{\V \in \Sigma^*\ :\ V_i \geq V_k \quad \forall\, k \in I \quad\mbox{and}\quad\exists\, k\neq i \ : \ V_i=V_k\}.
\end{equation}

Suppose now that $J\subset P(I^-)$ contains two or more neurons, and let $\V\in\Sigma^*_{J}$. From (\ref{SIGMAJINHIB}) it follows that $V_i\geq V_k$ for all $i\in J$  and $k\in I$, and $V_i=V_j$ for all $i,j \in J$. Therefore, $\V$ belongs to the intersection of the boundary of the sets $\Sigma^*_{\{i\}}$ such that $i\in J$.
We deduce that
\begin{equation}\label{J2BORDE}
\Sigma^*_{J}\subset \bigcup_{i\in J}\partial\Sigma^*_{\{i\}}\quad \forall J\subset P(I^-)\quad\text{such that}\quad\#J\geq 2.
\end{equation}

Finally, let us study the set $\Sigma^*_I$. According to (\ref{SIGMAI}), its interior $\overset{\circ}{\Sigma^*_{I}}$ is
\[
\overset{\circ}{\Sigma^*_{I}}=\{\V \in \Sigma^*\ : \ \max_{i\in I^+} V_i > V_k \quad \forall\, k \in I\},
\]
and its boundary is the subset of $\Sigma^*_I$ defined by
\begin{equation}\label{BORDESIGMAI}
\partial\Sigma^*_{I}=\{\V \in \Sigma^*\ :\
\max_{i\in I^+} V_i \geq V_k \quad \forall\, k\in I\quad\mbox{and}\quad\exists\, j\in I^- \ : \ V_j= \max_{i\in I^+} V_i\}.
\end{equation}
If $\V\in\partial\Sigma^*_{I}$, then there exists $j\in I^-$ such that $V_j \geq V_k$ for all $k\in I$, with equality for a $k\in I^+$ and thus different of  $j$. In other words $\V\in\partial\Sigma^*_{\{j\}}$ for a $j\in I^-$. It follows that
\[
\partial\Sigma^*_{I}\subset \bigcup_{i\in I^-}\partial\Sigma^*_{\{i\}}.
\]

To sum up, the previous analysis shows that $\Sigma^*$ can be decomposed as the union of $\#I^-+2$ pairwise disjoint sets: the open sets $\Sigma^*_{\{i\}}$ (one for each inhibitory neuron $i \in I^-$),
the interior of $\Sigma^*_{I}$, and the closed set with empty interior $\partial{\mathcal P}^*$.
This last set is the union of $\partial\Sigma^*_I$ with the sets $\Sigma^*_J$ such that $J$ contains only inhibitory neurons and $\#J\geq 2$. Moreover, it coincides with the union of the topological boundaries of the open sets $\Sigma^*_{\{i\}}$. So, we have
% This last set which is the union of the topological boundaries of the open sets $\Sigma^*_{\{i\}}$:
\begin{equation}\label{DECOMPSIGMA*}
\Sigma^*= \bigcup_{i\in I^-}\Sigma^*_{\{i\}} \bigcup \overset{\circ}{\Sigma^*_I}\bigcup\partial\mathcal{P}^*\qquad
\mbox{where}\qquad
\partial\mathcal{ P}^*:=\bigcup_{i\in I^-}\partial\Sigma^*_{\{i\}}.
%=\bigcup_{J\in P(I^-):#J>1}\Sigma^*_{J}\bigcup\partial\Sigma^*_I.
\end{equation}
Each set of the type $\Sigma^*_{\{i\}}$ consists of the initial states of the network for which, in the firing regime, only the inhibitory neuron $i$ fires. The set $\overset{\circ}{\Sigma^*_{I}}$ are the initial states of the networks for which only excitatory neurons (one or several) reach spontaneously  the threshold. According to Lemma \ref{JVAC}, the firing of these excitatory neurons   induces the synchronization of the network. Finally, for the initial states $\V\in\partial{\mathcal P}^*$, at least two neurons reach the threshold spontaneously and at least one is inhibitory. If they are all inhibitory, then only these neurons fire (it corresponds to the case where $\V\in\Sigma^*_J$, with $\#J>1$), but if one of them is excitatory, its firing induces the synchronization of the network
(it corresponds to the case where $\V\in\partial\Sigma^*_I$).

\vspace{1ex}
The return map is continuous in $\overset{\circ}{\Sigma^*_{I}}$ and in any $\Sigma^*_{\{i\}}$, because each of these sets is open, and contained in the interior of a set $\Sigma_J$ where $\rho$ is continuous (see (\ref{RETURN})). So, the boundary set $\partial{\mathcal P}^*$ contains all the discontinuity points of the return map. Let us give an intuitive idea explaining why those discontinuities appear. For an initial state $\V \in \partial{\mathcal P}^*$ at least two neurons fire in the firing regime, say $i\in I^-$ and $j\in I$. But, there exists an arbitrarily small perturbation $\W$ of $\V$ that belongs to $\Sigma^*_{\{i\}}$, i.e. such that only the neuron $i$ fires in the firing regime. This implies the existence of a gap between $\rho(\V)$ and $\rho(\W)$, since
$\rho_j(\V)=0$ and $\rho_j(\W)\neq 0$, except for a very specifics value of $H_{ij}$. The following lemma makes rigorous this intuitive idea.

\begin{Lemma}\label{lemacontinuidad1} If $\V\in {\partial{\cal P}^*}$, then  there exists a
sequence $\{\U^{m}\}_{m\in\N}$ of points of $\Sigma^*$ such that $\lim\limits_{m\to\infty}\U^m=\V$ and
$\lim\limits_{m\to\infty}\|\rho(\U^m) - \rho (\V)\| \geq\mu$ where
\begin{equation}\label{MU}
\mu:=\min\{|\alpha|,\min\limits_{i\neq j}|H_{ij}+\theta|\}.
\end{equation}
\end{Lemma}

\begin{proof} Let $\V\in\partial\mathcal{P}^*$ and $i\in I^-$ such that
$\V\in\partial\Sigma^*_{\{i\}}$. By definition of boundary, there is a sequence $\{\U^{m}\}_{m\in\N}$ of points of $\Sigma^*_{\{i\}}$ such that $\lim\limits_{m\to\infty} \U^m=\V$. For all $k\neq i$,
we have
\begin{eqnarray*}
\lim_{m\to\infty}\rho_k(\U^{m})
&=& \lim_{m\to\infty}\max\{\alpha, \; \beta - \frac{(\beta - U^{m}_k)( \beta -\theta)}
{\beta - U^{m}_i} + H_{ik} \}\\
&=& \max\{\alpha, \; \beta - \frac{(\beta- V_k)(\beta - \theta)}{\beta - V_i} + H_{ik}\}.
\end{eqnarray*}

On the other hand, if $\V\in\partial\Sigma^*_{\{i\}}$, it belongs to a set $\Sigma^*_J$ such that $i\in J$ and $\# J\geq 2$ (possibly $J=I$). Therefore, there is $j\neq i$ in $J$ such that $V_i=V_j$ and it follows that $\lim\limits_{m\to\infty}\rho_j(\U^{m})=\max\{\alpha,\; H_{ij} + \theta\}$. As $j\in J$ and $\V\in\Sigma^*_{J}$, we have $\rho_j(\V)=0$, and we obtain that
\[
\lim_{m\to\infty}\|\rho(\V)-\rho(\U^m)\|\geq \lim_{m\to\infty}|\rho_j(\V)-\rho_j(\U^m)|=|\max\{\alpha,\; H_{ij} +
\theta\}|
%\geq \min\{|\alpha|,\min_{i\neq j}|H_{ij}+\theta|\}.
\]
which leads to the desired inequality.
\qed
\end{proof}

To conclude this section we give a general definition of a piecewise contractive map and we show that the return map fulfill this definition.

% The proposition is true under the hypothesis (H1), (H2), (H3) and (H4), and the additional generic hypothesis $\min\limits_{i\neq j}|H_{ij}+\theta|\neq 0$. However, we do not include these hypothesis in its statement because it is also a general definition of a piecewise contractive map.

%To conclude this section, we sum up in the following proposition the important properties of the return map.
%The proposition is true under the hypothesis (H1), (H2), (H3) and (H4), and the additional generic hypothesis
%$\min\limits_{i\neq j}|H_{ij}+\theta|\neq 0$.

\begin{Definition}\label{DEFCONTRACTIVEMAP}{(Piecewise contractive map)} \em A map $\rho$, from a compact set $\Sigma^*\subset\R^n$ to itself, is said to be piecewise contractive  if it satisfies the following three conditions:

\vspace{1ex}
\noindent
1) There exists a finite family $\Sigma^*_{\{0\}},\Sigma^*_{\{1\}},\dots,\Sigma^*_{\{p\}}$ of pairwise disjoint open sets such that
\[
\Sigma^* = \bigcup\limits_{i=0}^p \overline{\Sigma^*_{\{i\}}}.
\]

\vspace{1ex}
\noindent 2) There exists a constant $\lambda<1$ such that for all $i \in \{0,1,\dots,p\}$
\[
\|\rho(\V)-\rho(\W)\|\leq \lambda\|\V-\W\|\quad\forall\, \V,\W\in \Sigma^*_{\{i\}}.
\]

\vspace{1ex}
\noindent 3) The map $\rho$ has a discontinuity jump larger than a constant $\mu>0$ in the borders
$\partial\Sigma^*_{\{i\}}$ of a the sets $\Sigma^*_{\{i\}}$. \em
\end{Definition}

Under the hypothesis (H1), (H2), (H3) and (H4), and the additional generic hypothesis
\[
\min\limits_{i\neq j}|H_{ij}+\theta|\neq 0,
\]
which stipulates that the smallest inhibitory interaction is different from $-\theta$, the return map of the network is piecewise contractive.

First, by Proposition \ref{INVARIANT}, the set $\Sigma^*$ is compact and $\rho(\Sigma^*)\subset\Sigma^*$. Second, if we denote $\Sigma^*_{\{0\}}$ the interior of $\Sigma^*_{I}$ and $p$ the number of inhibitory neurons, the decomposition (\ref{DECOMPSIGMA*}) of $\Sigma^*$ shows that the condition 1) is satisfied. Third, the condition 2) is also satisfied because $\Sigma^*$ is a contractive zone with respect to the partition $\mathcal{P}$, and thus it is also a contractive zone with respect to the partition $\mathcal{P}^*$. Finally, the Lemma \ref{lemacontinuidad1} shows that 3) is true for $\mu:=\min\{|\alpha|,\min\limits_{i\neq j}|H_{ij}+\theta|\}$,
since $\alpha<0$ by Hypothesis (H1) and $\min\limits_{i\neq j}|H_{ij}+\theta|\neq 0$.

\vspace{1ex}
The results of the sequel of this section apply to any abstract piecewise contractive map satisfying  Definition \ref{DEFCONTRACTIVEMAP}. So, they apply in particular  to the return map under the hypothesis mentioned above. To simplify the notations, and to fit with those of Definition \ref{DEFCONTRACTIVEMAP}, we will denote $\Sigma^*_{\{0\}}$ the interior of $\Sigma^*_I$. It does not mean that a new neuron $0$ is introduced in the network, it rather recall that if $\V\in\Sigma^*_{\{0\}}$ then $\rho(\V)=0$. Also, $I_0$ will denote the set
$\{0,1,\dots,\#I^-\}$.

\subsection{The stable and the sensitive sets.}\label{STABLANDCAOTIC}

In order to study the asymptotic dynamics of the return map, we divide the contractive zone $\Sigma^*$ in two complementary sets: the
stable set $S$ and the sensitive set $C=\Sigma^*\setminus S$.
%In order to study the asymptotic dynamics we divide the contractive zone $\Sigma^*$ in two complementary sets, the
%stable set $S$ and the sensitive set $C=\Sigma^*\setminus S$. The set $S$ is formed by the states whose future orbit changes continuously under small perturbations of the initial state. The set $C$ is formed by the states whose future orbit can change drastically under arbitrarily small perturbations of the initial state.

\begin{Definition}\label{STABLEANDDELTASTABLE}{(Stable set)} \em A point $\V\in\Sigma^*$ is  stable, if for all $
\nu>0$ there exists $\delta>0$ such that for all $p\in\N$:
\begin{equation}\label{DEFESTABLE}
\mbox{if}\quad \|\rho^p(\V) - \W\|<\delta  \quad\mbox{ then }\quad
\|\rho ^k(\rho^p(\V)) - \rho ^k(\W)\|<\nu \quad\forall\, k\geq 1.
\end{equation}
We call stable set and denote $S$ the set of all the stable points.
\end{Definition}
If $\V$ is a stable point, then any sufficiently small perturbation of $\V$, or of a point of its orbit, has an orbit which remains near the orbit of $\V$. For instance, if the initial state of a network is stable, one can expect that a small perturbation may modify slightly and temporary the inter-spike intervals, but would not change the firing neurons.

%For instance, one can expect that if the initial state of a network is stable, then a small perturbation may modify slightly and temporary its inter-spike intervals, but does not change the firing neurons.

\begin{Definition}\label{CHAOTICSET}{(Sensitive set)} \em A point $\V\in\Sigma^*$ is sensitive, if there exists $\nu>0$ such that for all $\delta>0$ there are
$p\in\N$ and $\W\in\Sigma_c$ satisfying:
\[
\|\rho^p(\V) - \W\|< \delta  \quad \mbox{ and }\quad \|\rho ^k(\rho^p(\V)) - \rho ^k(\W)\|\geq\nu \quad\mbox{for
some}\quad k\geq 1.
\]
We call sensitive set and denote $C$, the set of all the sensitive points.
\end{Definition}
The sensitive set is the set of the points that are not stable. For a network in a sensitive state, there exit arbitrarily small perturbations that produce a consequent change of its time evolution. Typically, such perturbations modify the firing neurons at an instant of its evolution.

\begin{Lemma}\label{teoremaSinvariante} The stable set $S$ is forward invariant, i.e $\rho(S) \subset S$, and
the sensitive set $C$ is backward invariant, i.e $\rho^{-1}(C) \subset C$.
\end{Lemma}

\begin{proof} Let us prove that $\rho(S) \subset S$. Take $\V \in S$. Then for all $\nu>0$ there exists $\delta>0$ such that for all $p\geq 1$
\[
\mbox{if}\quad\|\rho^{p-1}({\bf \rho(\V)}) -\W \| < \delta   \quad\mbox{ then }\quad \|\rho^k(\rho ^{p-1}
({\rho(\V)})) -\rho^k(\W)\| \leq \nu \qquad \forall k \geq 1.
\]
The last assertion is the definition of stable point applied to $\rho(\V)$. Therefore $\rho(S) \subset S$. As $C$ is the complement in $\Sigma^*$ of $S$, and $\rho(S)\subset S$  we have $\rho^{-1}(C) \subset C$. \qed
\end{proof}
Lemma \ref{teoremaSinvariante} proves that if a network is initially in a stable state then it is forever stable.
If it is in a sensitive state, its past states are sensitive, but its future states can be stable.

\vspace{1ex}
Now, we give a useful characterization of the stable set:

\begin{Lemma}\label{MORTAL} Let us denote $d$ the distance induced by the norm $\|\cdot\|$, that
is $d(\V,\W)=\|\V-\W\|$ for all $\V$ and $\W$ in $\Sigma^*$. For all $\eta>0$, let
$S_\eta$ be the set of the stable points whose orbit is at a distance larger than $\eta$ of $\partial\mathcal P^*$:
\[
S_\eta:=\{\V\in S\ :\ d(\rho^p(\V),\partial\mathcal{P}^*)\geq\eta\quad\forall\,p\in\N \}
\qquad\forall\eta>0
\]
then,
\begin{equation}\label{DECOMPS}
S=\bigcup_{\eta>0}S_\eta.
\end{equation}
\end{Lemma}
\begin{proof} By definition, $S_{\eta}\subset S$ for any $\eta>0$. Now, let $\V\in S$ and suppose by contradiction that $\V\notin S_\eta$ for all $\eta>0$. Let $\delta>0$ and $\mu$ be the discontinuity jump of $\rho$ in $\partial\mathcal{P}^*$ defined in (\ref{MU}).
As $\V$ is a stable point, there exits $0<\delta_0<\delta$ such that for all $p\in\N$
\[
\text{if}\quad \|\rho^p(\V)-\U\|<\delta_0\quad\text{then}\quad \|\rho(\rho^p(\V))-\rho(\U)\|<\frac{\mu}{4}.
\]
As $\V\notin S_{\delta_0/2}$, there exist $p_0$ and $\W\in\partial\mathcal{P}^*$ such that
\[
\|\rho^{p_0}(\V)-\W\|<\frac{\delta_0}{2}<\delta.
\]
By Lemma \ref{lemacontinuidad1}, there exist $\U\in\Sigma^*$  such that
\[
\|\W-\U\|<\frac{\delta_0}{2}\quad\text{and}\quad \|\rho(\W)-\rho(\U)\|>\frac{\mu}{2}.
\]
For this $\U$, we have $\|\rho^{p_0}(\V)-\U\|\leq \|\rho^{p_0}(\V)-\W\|+\|\W-\U\|<\delta_0$. Thus,
\[
\|\rho(\rho^{p_0}(\V))-\rho(\W)\|\geq \left|\|\rho(\W)-\rho(\U)\|-\|\rho(\rho^{p_0}(\V))-\rho(\U)\| \right|
\geq \left|\frac{\mu}{2}-\frac{\mu}{4}\right|=\frac{\mu}{4}.
\]
To sum up, we have shown the existence of a $\nu:=\mu/4$ such that for an arbitrary $\delta>0$, there exists $p_0$ and $\W$ such that
\[
\|\rho^{p_0}(\V)- \W\|<\delta\quad\text{and}\quad \|\rho(\rho^{p_0}(\V))-\rho(\W)\|\geq\nu.
\]
It follows that $\V$ belongs to the sensitive set, which is a contradiction. We conclude that for all $\V\in S$ there exists $\eta>0$ such that $\V\in S_\eta$.\qed
\end{proof}
Note that each set $S_\eta$ is forward invariant and does not contain point of $\partial\mathcal{P}^*$. In other words,
\begin{equation}\label{INCLUSIONSNUDELTA}
\rho(S_\eta)\subset S_\eta\quad\text{and}\quad S_{\eta}\subset\bigcup\limits_{i\in I_0}\Sigma^*_{\{i\}}\quad\forall\,\eta>0.
\end{equation}
and by (\ref{DECOMPS}) the same relations are true for $S$. It follows in particular that $\rho$ is continuous in any $S_{\eta}$ and in $S$. Also $\partial\mathcal{P}^*$ is a subset of the sensitive set. The relations (\ref{INCLUSIONSNUDELTA}) show that if the initial state of the network is stable, then either it get synchronized, if its orbit visit $\Sigma^*_{\{0\}}$, or only inhibitory neurons fire and never at the same time. The situation where various inhibitory neurons fire simultaneously occurs only when the network is in a sensitive state.

%%%%%%%%%%%%%%%%%%%%%%%%%%%%%%%%%%%%%%%%%%%%%%%%%%%%%%%%%%%%%%%%%%%%%%%%%%%%%%%%%%%%%%%%%%%%%%%%%%%%%%%%%%%%%%%%%%

\subsection{Principal results on the asymptotic dynamics}\label{PERIODSYST}

%%%%%%%%%%%%%%%%%%%%%%%%%%%%%%%%%%%%%%%%%%%%%%%%%%%%%%%%%%%%%%%%%%%%%%%%%%%%%%%%%%%%%%%%%%%%%%%%%%%%%%%%%%%%%%%%%%

In this subsection we prove Theorem \ref{teoremaPrincipal} which states that the limit sets attracting the stable points is only composed of limit cycles. The number of limits cycle can be finite or infinite, but it is always countable. The proof essentially relies on the contraction property of the return map. Since the map is not continuous, it is not possible to apply directly a classic fixed point theorem for contractive maps. Nevertheless, we give a generalization of the Banach fixed point theorem for piecewise contractive maps. It states the existence of periodic orbits - maybe of periods different from one and not necessarily unique - that attract all the stable points.

\begin{Definition}{($\omega$-limit set)}\label{DEFOMEGALIM} \em The $\omega$-limit set $\omega(\V)$
of a point $\V\in\Sigma^*$ is the set of the limit points of the future orbit of $\V$, that is:
\[
\omega(\V) = \{\W \in \Sigma^* \ :\  \exists\, {\{p_k\}}_{k\in\N} \ :\  \lim_{k\to\infty} p_k = +\infty \mbox{ and }
\lim_{k\to\infty}\rho^{p_k}(\V)=\W\}.
\]
For a set $A\subset\Sigma^*$, we denote $\omega(A)$ the union of all the $\omega$-
limit sets of the points of $A$. That is,
\[
\omega(A):=\bigcup_{\V\in A}\omega(\V).
\]
\em
\end{Definition}
In other words, $\W$ belongs to the $\omega$-limit set of $\V$ if there exists a subsequence of the orbit
of $\V$ which converges to $\W$. As the phase space $\Sigma^*$ is compact, the $\omega$-limit set of any point is nonempty. Intuitively, the $\omega$-limit set of $\V$ is the set of the points of $\Sigma^*$ for which there exists an arbitrarily small neighborhood which is always visited by the orbit of $\V$. In this sense, the points of $\omega(\V)$ are those which attract the orbit of $\V$. The set $\omega(\V)$ is forward invariant if $\rho$ is continuous in $\omega(\V)$, which is not necessarily the case for any $\V$. Nevertheless, the $\omega$-limit set is the same for all the points of a same orbit.

\begin{Definition} {(Limit cycle)} \em A set $L \subset\Sigma^*$ is a limit cycle, if it is a periodic orbit whose basin of attraction
\[
\mathcal{B}(L) = \{\W \in \Sigma^*\ :\ \omega (\W) = L\}
\]
contains an open neighborhood of $L$.
\end{Definition}
% (also called a sink in the case of the discrete dynamics of the Poincar\'{e} map)

Now we state the two important results of this section:

\begin{Proposition}\label{CYCLESETA} For any $\eta>0$ such that $S_\eta\neq\emptyset$, the set
$S_\eta$ contains only a finite number $N(\eta)$ of limit cycles, and any point of $S_\eta$
belongs to the basin of attraction of one of these limit cycles. In other words, for all $\eta>0$
%\[
%\omega(S_{\eta}):=\bigcup_{\V\in S_{\eta}}\omega(\V)=\bigcup_{i=1}^{N(\eta)} L_i,
%\]
\[
\omega(S_{\eta})=\bigcup_{i=1}^{N(\eta)} L_i,
\]
where each set $L_i$ is a limit cycle included in $S_{\eta}$.
\end{Proposition}
We postpone the proof of the proposition to Section \ref{PROOFCYCLESETA}. This proposition allow us to deduce the following theorem which is a more detailed statement of the part 2) of Theorem \ref{Theorem1}. As well as Proposition \ref{CYCLESETA}, this theorem is true for any piecewise contractive map satisfying Definition \ref{DEFCONTRACTIVEMAP}.

\begin{Theorem} \label{teoremaPrincipal} If $S\neq\emptyset$, any stable point belongs to the basin of attraction of a limit cycle contained in the stable set. The number of these limit cycles is countable. In other words,
%\[
%\omega(S):=\bigcup_{\V\in S}\omega(\V)=\bigcup_{i=1}^N L_i,
%\]
\[
\omega(S)=\bigcup_{i=1}^N L_i,
\]
where each set $L_i$ is a limit cycle included in $S$. Moreover, the number $N$ of limit cycles is finite if $d(\omega(S),\partial\mathcal{P})>0$ and it is infinite, but countable, if $d(\omega(S),\partial\mathcal{P})=0$.
\end{Theorem}
\begin{proof} By Lemma \ref{MORTAL} any stable point $\V$ belongs to a $S_\eta$ for some $\eta>0$. Therefore, by Proposition \ref{CYCLESETA}, it belongs to the basin of attraction of a limit cycle in $S_\eta\subset S$, that is $\omega(\V)$ is a limit cycle contained in $S$. As for all $\eta'>\eta$ we have $S_{\eta'}\subset S_\eta$ we can rewrite (\ref{DECOMPS}) as
\[
S=\bigcup_{k\in\N}S_{\frac{l}{k}}\quad\text{where}\quad l:=\max_{\V\in\Sigma^*}d(\V,\partial\mathcal{P}^*).
\]
Since for each $k\in\N$ the set $\omega(S_{l/k})$ is a finite union of limit cycles and $S$ is a countable union of the sets $S_{l/k}$, it follows that $\omega(S)$ is a countable union of limit cycles.

\vspace{1ex}
If $d(\omega(S),\partial\mathcal{P}^*)>0$, then there exists $\eta>0$ such that $d(\omega(S),\partial\mathcal{P}^*)\geq\eta$. Since the points of $\omega(S)$ are stable, the return map is continuous in $\omega(S)$ and therefore $\omega(S)$ is invariant. It follows that the orbit of any  point of $\omega(S)$ remains at a distance larger than $\eta$ of $\partial\mathcal{P}^*$. In other words, $\omega(S)\subset S_{\eta}$. By Proposition \ref{CYCLESETA} it implies that $\omega(\omega(S))$
is a finite union of limit cycles. But as $\omega(S)$ contains only periodic points we have
$\omega(\omega(S))=\omega(S)$. This show that if $d(\omega(S),\partial\mathcal{P}^*)>0$ then $\omega(S)$ is a finite union of limit cycle.

\vspace{1ex}
Now suppose $d(\omega(S),\partial\mathcal{P}^*)=0$. Then, either $\omega(S)\cap\partial\mathcal{P}^*\neq\emptyset$, or there exits an infinite sequence of different points of $\omega(S)$ which converges to a point of $\partial\mathcal{P}^*$. The first option is impossible since $\omega(S)\subset S$ and $\partial\mathcal{P}^*$ is a subset of the sensitive set. Therefore, the second option holds and $\omega(S)$ is infinite. This ends the proof, since in the other hand, we know that is it countable.
\qed
\end{proof}

%\begin{Theorem} \label{teoremaPrincipal} If $S\neq\emptyset$, any stable point belongs to the basin of attraction of a limit cycle contained in the stable set. The number of limit cycles contained in the stable set is countable, and it is finite if the stable set is at positive distance of $\partial\mathcal{P}^*$.
%\end{Theorem}
%\begin{proof} By Lemma \ref{MORTAL} any stable point belongs to a $S_\eta$ for some $\eta>0$. Therefore, by Proposition \ref{CYCLESETA}, it belongs to the basin of attraction of a limit cycle in $S_\eta\subset S$. As for all $\eta'>\eta$ we have $S_{\eta'}\subset S_\eta$ we can rewrite (\ref{DECOMPS}) as
%\[
%S=\bigcup_{k\in\N}S_{\frac{l}{k}}\quad\text{where}\quad l:=\max_{\V\in\Sigma^*}d(\V,\partial\mathcal{P}^*).
%\]
%Since each $S_{\frac{l}{k}}$ contains a finite number of limit cycle and $S$ is a countable union of these sets, the sable set contains only a countable number of limit cycles. Finally, if $d(S,\partial\mathcal{P}^*)>0$, then there exists $\eta>0$ such that $d(S,\partial\mathcal{P}^*)\geq\eta$, which implies that $S=S_\eta$.
%Therefore, in this case, $S$ contains a finite number of limit cycles. \qed
%\end{proof}

Now we detail the consequences of the previous results on the dynamics of IF neural networks.

Theorem \ref{teoremaPrincipal} states that if the initial state $\V$ of a network is stable (i.e. $\V\in S$) then, when time tends to infinity, the network tends to fire periodically, since it is attracted by a limit cycle. If an excitatory neurons fires, this limit cycle is a synchronized orbit. But if no excitatory neuron fires, then the asymptotic firing patterns of each neuron present a common period but they are not synchronized (recall (\ref{INCLUSIONSNUDELTA}), two inhibitory neurons cannot fire at the same time). According to the initial stable state of the network, the attracting limit cycle can change. The stable set contains only a countable number of limit cycles and certainly, for most values of the parameters, this number is finite. However, it can also be infinite.

The situation described above does not prevent the stable dynamics of the network from exhibiting a behavior which seems chaotic, in a wide sense of the term. By definition of the stable set, any sufficiently small perturbation is damped and does not modify the asymptotic periodic pattern of the network. Nevertheless, the size of such perturbation depends on the set $S_{\eta}$ to which  the initial state of the network belongs. If $\eta$ is large, then large perturbations are
allowed. But if $\eta$ is small, then the orbit of $\V$ can get near to the discontinuity set $\partial\mathcal{P}^*$, or more generally, near to a sensitive point. in such a case some small perturbations (but non arbitrarily small) can change significantly the behavior of the network. The most probable effect is a change of the limit cycle attracting the network. Also, the network enters into a transitory regime before reaching this new limit cycle.

Moreover, the transitory time before the network gets near to a limit cycle and the period of this
limit cycle cannot decrease when $\eta$ decreases (see Remark \ref{TRANSITCYCLE}). In other words, the initial stable states whose orbits get the closer to the sensitive points have the larger transitory times. Besides, the limit cycles attracting them have the largest periods. It follows that those limit cycles may never be observed during a finite time of experimentation, since the network is in a transient regime and/or it tends to a limit cycle of very large period. This phenomenon  most probably occurs if the distance between $\omega(S)$ and $\partial\mathcal{P}^*$ is equal to zero. In such a case $\omega(S)$ contains an infinite number of limit cycles, and for any given number $\delta>0$, there is an infinity of limit cycles at a distance smaller than $\delta$ of $\partial\mathcal{P}^*$. Thus, most of the initial states of the network have an orbit which is attracted by one of these  limit cycles of large period. Furthermore, if a small perturbation is applied repeatedly to one of these orbits, for example by rounding errors in numerical simulations or in the presence of noise, it can maintain the orbit in the same region, but change many times the limit cycles it is attracted by.

\vspace{1ex}
Theorem \ref{teoremaPrincipal} and Proposition \ref{CYCLESETA} are results about the asymptotic dynamics of the stable set. However, Lemma \ref{MORTAL} allows us to make some comments about the dynamics of the complementary set, namely the sensitive set.
This lemma proves that any stable point belongs to a set $S_\eta$ for some $\eta>0$. Therefore, all the points of the orbit of a stable point $\V$ are at a distance larger than $\eta(\V)>0$ of $\partial\mathcal{P}^*$. It follows that the sensitive set consists of the points whose orbit is arbitrarily near to $\partial\mathcal{P}^*$. These points are those of $\partial\mathcal{P}^*$ and of all its pre-images $\bigcup\limits_{k\in\N}\rho^{-k}(\partial\mathcal{P}^*)$, as well as the points whose $\omega$-limit set contains a point of $\partial\mathcal{P}^*$. Formally,
\[
C=\bigcup\limits_{k\in\N}\rho^{-k}(\partial\mathcal{P}^*)\bigcup\mathcal{C}\quad\text{where}\quad
\mathcal{C}:=\{\V\in\Sigma^*\ :\ \omega(\V)\cap\partial\mathcal{P}^*\neq\emptyset\}.
\]
The points of $C$ that are not in $\mathcal{C}$ have an orbit which is eventually stable and Theorem \ref{teoremaPrincipal} applies for them after a transitory time. The set $\mathcal{C}$ is invariant, and for any
initial state $\V$ in this set, there exit arbitrarily small perturbations that produce a large deviation in its orbit. The invariance of $\mathcal{C}$ and its sensitivity to arbitrary small perturbations are ingredients of the concept of chaos. However, they are not enough to ensure the presence of chaos from a strict mathematical point of view. In particular, the existence of piecewise contractive maps exhibiting chaotic attractors is a widely open question. Also, the exact nature of the possible attractors of $\mathcal{C}$ has been few documented.
We can nevertheless mention \cite{CGMU}, where this question is investigated for piecewise affine contractive maps.
In this paper, it is shown that the attractor of $\mathcal{C}$ can have any type of cardinality (finite, countable infinite and uncountable). In particular, the most complex found attractor is a transitive union of a Cantor set and an infinite countable set, which does not contain any periodic orbit.

\vspace{1ex}
From the results of this section, it follows that the piecewise contractive character of the return map leads the stable dynamics to be asymptotically periodic. However, it does prevent the network from exhibiting a dynamics that can seem chaotic when observed at finite time scales. The richness of this dynamics depends on the interrelations between the stable set and the chaotic set. If the asymptotic sets of the stable dynamics remain far from the discontinuities of the return map, then numerical simulations should observe a stable periodic behavior. But, in the other case, this periodic behavior may be not observed at finite time scales and/or in presence of perturbations, for an important set of initial states of the network. Also, non periodic attractors exist in the sensitive set and can coexist with the periodic attractors of the stable set.

\subsection{Proof of Proposition \ref{CYCLESETA}}\label{PROOFCYCLESETA}
To prove Proposition \ref{CYCLESETA} we introduce the so-called \em atoms \em of a set $S_\eta$. For all $i\in I_0$, we define $F_i: P(\Sigma^*)\to P(\Sigma^*)$, where $P(\Sigma^*)$ denotes the set of all the subsets of $\Sigma^*$, by:
\[
F_i(E)=\rho(E\cap\Sigma^*_{\{i\}}\cap S_\eta)\qquad \forall\, E\subset\Sigma^*.
\]
Given $k\in\N$ and $(i_1,i_2,\dots,i_k)\in I_0^k$, we call {\it atom of generation $k$} the set
\[
A_{i_1i_2\dots i_k}=F_{i_k}\circ F_{i_{k-1}}\circ\dots\circ F_{i_1}(\Sigma^*_{\{i_1\}})
\]
and we call {\it family of the atoms of generation $k$} the set
$\mathcal{A}_k=\{A_{i_1i_2\dots i_k},\,(i_1,i_2,\dots,i_k)\in {I_0}^{k}\}$. Note that the forward
invariance of $S_\eta$ by $\rho$ ensures that any atom is contained in $ S_\eta$.

\begin{Lemma}\label{PROPATOMS} For all $k\in\N$, we have

\noindent $i)$ If $\V\in S_\eta$, then $\rho^k(\V)$ belongs to an atom of generation $k$.

\noindent $ii)$ Any atom of generation $k+1$ is contained in an atom of generation $k$.

\noindent $iii)$ Denotes $\mbox{\em diam}(A)$ the diameter of $A$ and let $d_k=\max\limits_{A\in\mathcal{A}_k}\mbox{\em diam}(A)$. Then $\lim\limits_{k\to\infty} d_k=0$.
\end{Lemma}

\begin{proof}$i)$ It follows from the forward invariance of $S_\eta$  and
from the inclusion $S_\eta\subset\bigcup\limits_{i\in I_0}\Sigma^*_{\{i\}}$ stated
in (\ref{INCLUSIONSNUDELTA}). Assume that there exists $\V\in S_\eta$. Then, there  exists   $i_1\in I_0 $ such that $\V\in \Sigma^*_{\{i_1\}}\cap S_\eta$. It follows
that
\[
\rho(\V)\in\rho(\Sigma^*_{\{i_1\}}\cap S_\eta)=F_{i_1}(\Sigma^*_{\{i_1\}})\in\mathcal{A}_1.
\]
By induction, suppose  that $\rho^k(\V)$ belongs to $A_{i_1\dots i_k}\in\mathcal{A}_k$. From (\ref{INCLUSIONSNUDELTA}) it follows that $\rho^{k}(\V)\in \Sigma^*_{\{i_{k+1}\}}\cap S_\eta$
for some $i_{k+1}\in I_0$. Using the induction hypothesis we obtain
\[
\rho^{k+1}(\V)\in\rho(\Sigma^*_{\{i_{k+1}\}}\cap S_\eta\cap A_{i_1\dots i_k})
%=\rho(\Sigma_{c}^{i_{k+1}}\cap S_\eta\cap F_{i_k}\circ\cdots\circ F_{i_1}(\Sigma_{c}^{i_1}))
=F_{i_{k+1}}\circ\cdots\circ F_{i_1}(\Sigma^*_{\{i_1\}})\in\mathcal{A}_{k+1}.
\]

\noindent $ii)$ Let $A_{i_1 i_2\dots i_{k+1}}$ be an atom of generation $k+1$. Then,
\begin{eqnarray*}
A_{i_1 i_2\dots i_{k+1}}&=&F_{i_{k+1}}\circ\dots\circ F_{i_{2}}\circ F_{i_1}(\Sigma^*_{\{i_{1}\}})\\
&\subset& F_{i_{k+1}}\circ\dots\circ F_{i_2}(\Sigma^*)\\
&\subset& F_{i_{k+1}}\circ F_{i_{k}}\circ\dots\circ F_{i_2}(\Sigma^*_{\{i_{2}\}})
\in\mathcal{A}_{k}.
\end{eqnarray*}

\noindent $iii)$ It is enough to show by induction that
\begin{equation}\label{RECU}
\mbox{diam}(A)\leq \lambda^{k-1}\mbox{diam}(\Sigma^*) \qquad \forall\, A\in\mathcal{A}_k.
\end{equation}
Since any atom is a subset of $\Sigma^*$, it is true for $k=1$. Assume (\ref{RECU}) is true
for $k=k_0$. Take $A=A_{i_1i_2\dots i_{k_0+1}}\in\mathcal{A}_{k_0+1}$ and $\V$, $\W\in A$.
Then $A=\rho(A'\cap\Sigma^*_{\{i_{k_0+1}\}}\cap S_\eta)$ where $A'=A_{i_1i_2\dots i_{k_0}}\in
\mathcal{A}_{k_0}$.
Therefore, there exists $\V'$, $\W'\in\Sigma^*_{\{i_{k_0+1}\}}\cap A'$ such that $\V=\rho(\V')$
and $\W=\rho(\W')$. Applying the contraction property and the induction hypothesis we obtain:
\[
\|\V-\W\|=\|\rho(\V')-\rho(\W')\|\leq\lambda\|\V'-\W'\|\leq\lambda\mbox{diam}(A')\leq
\lambda^{k_0}\mbox{diam}(\Sigma^*),
\]
which implies the desired result, since $\V$ and $\W$ are arbitrary in the atom $A$. \qed
\end{proof}

\begin{Lemma}\label{INUTIL} There exists $k\geq 1$ such that if $A\in\mathcal{A}_k$ then $\overline{A}\subset\Sigma^*_{\{i\}}$ for some $i\in I_0 $.
\end{Lemma}

\begin{proof} Let $k\geq 1$ such that $d_k<\eta/2$ and $A\in\mathcal{A}_k$ not empty. Take $\V\in A$ and $i\in I_0 $ such that $\V\in\Sigma^*_{\{i\}}$ (recall $A\subset S_\eta$ and (\ref{INCLUSIONSNUDELTA})). Let $E=\Sigma^*\setminus\Sigma^*_{\{i\}}$ and let $\partial E$
be its boundary. We denote $d(\V,E)$ the distance between the point $\V$ and the set $E$. We have:
\[
d(\V,E)=d(\V,\partial E)=d(\V,\partial\Sigma^*_{\{i\}})\geq d(\V,\partial\mathcal{P}^*)\geq \eta.
\]
Let $\W\in\overline{A}$ then $d (\W,\V)\leq\mbox{diam }\overline{A}\leq d_k<\eta/2$ and from
\[
d(\W,E)\geq d(\V,E)- d(\W,\V)\geq \eta-\frac{\eta}{2}>0
\]
we deduce that $\W\notin E$. Then $\W\in\Sigma^*_{\{i\}}$. \qed
\end{proof}

\begin{Remark}\label{TRANSITCYCLE}\em The number of iteration $k$ of the return map that ensure that any atom of generation $k$ of $S_{\eta}$ is included in a continuity piece of the return map depends on $\eta$. This number is
an upper bound on the transitory time for all the orbits of $S_{\eta}$ to enter in the attraction basin of a limit cycle and on the period of this cycle. When $\eta$ increases this upper bound decreases. The largest transitory times and the largest period are therefore those of the points of the sets $S_{\eta}$ with the smallest $\eta$. \em
\end{Remark}

\begin{Lemma}\label{FANTOM} If there exist $(i_0,i_1,\dots,i_{p-1})\in I^p_0$ and a family of sets
$B_{0},\, B_{1},\dots,B_{p-1}$ satisfying

\vspace{1ex}
\noindent $i)$  $\overline{B}_{k}\subset\Sigma^*_{\{i_k\}}$ for all $k\in\{0,\dots,p-1\}$ and

\vspace{1ex}
\noindent $ii)$ $\rho(B_{p-1})\subset B_0$ and $\rho(B_{k-1})\subset B_{k}$ for all $k\in\{1,\dots,p-1\}$,

\vspace{1ex}
\noindent then there exists a unique periodic point of period $p$ in $\overline{B}_0$ whose orbit is the $\omega$-
limit set of any point contained in the union of the $B_k$'s.
\end{Lemma}

\noindent{\it Proof:} Since by $i)$ each $\overline{B}_k$ is contained in a contractive piece of $\rho$, by $ii)$ the map $ \rho^p$ is contractive in $\overline{B}_0$ and $\rho^{p}(\overline{B}_0)\subset \overline{B}_0$. Then, by the fixed point theorem of Banach,  there exists a unique periodic point $\tv$ of period $p$ in $\overline{B}_0$.

We prove now that if $\V\in B_{k}$, then $\omega(\V)$ is the orbit $L$ of $\tv$. Without loss of generality we can
assume $\V\in B_0$. It is enough to show that $\omega(\V) = \omega (\tv)$,
since the $\omega$-limit set of a periodic point coincides with its orbit.

Let ${\{p_j\}}_{j\in\N}$ be a sequence of natural numbers such that $\lim\limits_{j\to\infty} p_j=+\infty$ and such that either $\lim\limits_{j \to\infty}\rho^{p_j}(\V)$ or
$\lim\limits_{j \to\infty}\rho^{p_j}(\tv)$ exists. Since $\V$ and $\tv$ belong both to $\overline{B}_0$, according to  $ii)$ $\rho^{p_j}(\V)$ and $\rho^{p_j}(\tv)$ belong to the same
continuity piece for all $j\in\N$. Using the contraction property we obtain:
\[
\lim_{j\to\infty}\|\rho^{p_j}(\V) - \rho^{p_j}(\tv)\|
\leq\lim_{j\to\infty}\lambda^{p_j}\| \V  -  \tv \|=0.
\]
It follows that both $\lim\limits_{j\to\infty}\rho^{p_j}(\V)$ and $\lim\limits_{j\to\infty}\rho^{p_j}(\tv)$ exist
and are equal.
This proves that $\omega(\V) = \omega (\tv)$ as wanted. \CQFD

\vspace{.3cm}

\noindent{\it Proof of Proposition \ref{CYCLESETA}:} {\bf Step 1} Let ${\tilde k}\in\N$ be such that the thesis of Lemma \ref{INUTIL} is true. We show that the image by $\rho$ of any atom of $\mathcal{A}_{\tilde k}$ is contained in an atom of $\mathcal{A}_{\tilde k}$. Suppose $A\in\mathcal{A}_{\tilde k}$. By Lemma \ref{INUTIL},
there exists $i\in I_0 $ such that $\overline A\subset\Sigma^*_{\{i\}}$, and since any atom
is contained in $S_\eta$, we have $\rho(A)=\rho(A\cap\Sigma^*_{\{i\}}\cap S_\eta)
=F_{i}(A)$. Then, $\rho(A)\in\mathcal{A}_{\tilde k +1}$, and according to $ii)$ of Lemma \ref{PROPATOMS} it
is included in an atom of $\mathcal{A}_{\tilde k}$.

\noindent{\bf Step 2} Let $\V\in S_\eta$ and $\W=\rho^{\tilde k}(\V)$. Then, from $i)$ and
$ii)$ of Lemma \ref{PROPATOMS} we deduce that each point of the orbit of $\W$ belongs to an atom
of $\mathcal{A}_{\tilde k}$. Being the number of atoms in $\mathcal{A}_{\tilde k}$ finite, it must exist $B_0\in\mathcal{A}_{\tilde k}$ such that $\rho^{j_0}(\W)\in B_0$ and $\rho^{j_0+p}(\W)\in B_0$ for some $j_0$ and $p\in\N$. Let's denote $B_1,B_2,\dots,B_{p-1}$ the atoms of $\mathcal{A}_{\tilde k}$ that contain respectively
$\rho^{j_0+1}(\W),\rho^{j_0+2}(\W),\dots, \rho^{j_0+p-1}(\W)$. By Step 1, the image by $\rho$ of each
of these atoms is contained in an atom of $\mathcal{A}_{\tilde{k}}$, so they must satisfy:
\[
\rho(B_{p-1})\subset B_0 \quad\mbox{ and }\quad \rho(B_{k-1})\subset B_{k} \qquad\forall\,k\in\{1,\dots,p-1\}.
\]
Moreover, by Lemma \ref{INUTIL} there exists $(i_0,\dots,i_{p-1})\in I_0^{p}$ such that $\overline{B}_{k}
\subset\Sigma_{c}^{i_k}$ for all $0\leq k\leq p-1$. Therefore, this family of atoms satisfy the hypothesis of
Lemma \ref{FANTOM}. Since $\rho^{\tilde{k}+j_0}(\V)\in B_0$ and
$\omega(\V)=\omega(\rho^{{\tilde k}+j_0}(\V))$, we conclude that $\omega(\V)$ is a periodic orbit of period $p$
contained in the union of the atoms $B_k$'s. Note that it may exit at most $\#\mathcal{A}_{\tilde k}$ different
families satisfying Lemma \ref{FANTOM} and thus at most $\#\mathcal{A}_{\tilde k}$ different periodic orbits in $S_\eta$.
% au plus ils sont de periode 1, si il existe un point de periode 2 alors, il y a moins d'orbite periodique
% que si ils sont tous de periode 1 car celle ci occupe deja 2 atomes.

\noindent{\bf Step 3} We have shown that the $\omega$-limit set of any point of $S_\eta$
is a periodic orbit contained in $S_\eta$. We finish the proof of the propostion by
showing that any periodic orbit in $S_\eta$ is actually a limit cycle.

Suppose $\tilde{\V}$ is a periodic point of period $p$ and let
$L=\{\tilde{\V},\rho(\tilde{\V}),\dots,\rho^{p-1}(\tilde{\V})\}\subset S_\eta$ be its
orbit. We have to prove that there exists an open neighborhood of $L$ which points have $L$ as $\omega$-limit set
(i.e, the basin of attraction of $L$ contains an open neighborhood of $L$). For all $k\in\{0,\dots, p-1\}$ we have $\rho^k(\tv)\in\Sigma^*_{\{i_k\}}$ where $(i_1,\dots,i_p)\in I_0^{p}$. Since the continuity piece $\Sigma^*_{\{i_k\}}$ is open, there exists an open ball $B(\rho^k(\tv),a_k)$ of center $\rho^k(\tv)$ and radius  $a_k>0$ whose closure is contained in $\Sigma^*_{\{i_k\}}$.
Let $0<a<\min\{a_k,\ 0\leq k\leq p-1\}$ and $B$ be the open neighborhood of $L$ defined by:
\[
B=\bigcup_{k=0}^{p-1} B_k\quad \mbox{ where }\quad B_k=B(\rho^k(\tv),a)\quad \forall\, k\in\{0,\dots, p-1\}.
\]
We have $\overline{B}_k\in\Sigma^*_{\{i_k\}}$, and if we show that $\rho(B_{p-1})\subset B_0$ and $\rho(B_{k-1})
\subset B_k$
for all $k\in\{0,\dots, p-1\}$, then we can apply Lemma \ref{FANTOM} to prove that the open set $B$ contains
a unique periodic orbit ($L$) which is the $\omega$-limit set of all the points of $B$.
We will only prove that $\rho(B_{p-1})\subset B_0$, the other inclusions admit an analogous
proof. If $\V\in B_{p-1}$ then $\V$ and $\rho^{p-1}(\tv)$ belong to the same piece of continuity
$\Sigma^*_{\{i_{p-1}\}}$. Applying the contraction property of $\rho$ we obtain:
\[
\|\rho(\V)-\rho^{p}(\tv)\|\leq\lambda\|\V-\rho^{p-1}(\tv)\|<\lambda a<a.
\]
In other words $\|\rho(\V)-\tv\|<a$ which implies $\rho(\V)\in B_0$. \CQFD

%%%%%%%%%%%%%%%%%%%%%%%%%%%%%%%%%%%%%%%%%%%%%%%%%%%%%%%%%%%%%%%%%%%%%%%%%%%%%%%%%%%%%%%%%%%%%%

%%%%%%%%%%%%%%%%%%%%%%%%%%%%%%%%%%%%%%%%%%%%%%%%%%%%%%%%%%%%%%%%%%%%%%%%%%%%%%%%%%%%%%%%%%%%%%%%%%%%%%

\section{Conclusion}

In the present paper, we gave a mathematical analysis of the dynamics of IF neural networks, with
an arbitrary number of neurons interacting by instantaneous excitations and inhibitions. The main
purpose was to give a global analysis describing the asymptotic behavior of all the solutions of
the model.

Our study is in part motivated by the comparison of continuous time IF neural networks with discrete time
versions, such as those analyzed in \cite{C08}. An important characteristics of the model of \cite{C08} is
that it is piecewise contractive in the whole the phase space. In our case where time is continuous,
we show that the associated return map does not fulfill this property trivially. In particular, we found sets
of parameters where the return map exhibits expansion and has a repulsive periodic orbit.
It is therefore necessary to impose some conditions on the parameters values for the return map to be piecewise contractive in the whole phase space. We have proposed such conditions through the hypothesis (H3), which states that the interactions have to be sufficiently strong, or equivalently, the resistance of the neural membrane $R$ and/or the external current $I_{ext}$ sufficiently weak. However, since it is a non-trivial task to find parameters for which the return map exhibits real expansion, there certainly exist other sets of parameters than those defined by (H3), for which the return map is piecewise contracting in the whole phase space.

The property of piecewise contraction is mainly relevant to analyze the behavior of networks containing inhibitory neurons. We also study the dynamics of networks exclusively composed of excitatory neurons. For these networks, we only assume that all the interactions are strictly positive. We show that if the number of neurons is sufficiently large then, for any initial state, the network gets synchronized in a finite time. These results complete the work of \cite{MS90}, since we allow the weights of the excitatory synapses to be mutually different, while in \cite{MS90} all the interactions are equal to a same constant. The number of neurons that are necessary to provoke this global synchronization, as well as the duration of the transient regime, is a decreasing function of the minimum positive value of the interactions. We show in concrete examples, that if the number of neurons is not large enough, then the  global synchronization does not occur, since the network also presents some non synchronized periodic orbits. The relevance of the number of neurons (connections) for the synchronization has also been observed in \cite{DP08} where the effectiveness of the connections between neurons is modeled by a random variable. The synchronization is observed for high probability of connection. But if this probability is small, then asynchronous states appear.

%Finally, the synchronized periodic state is structurally stable (i.e. the synchronized state persists in the deterministic system under small perturbations of the parameter values); and also stochastically stable (i.e. it persists if noise is added to each iteration, provided that the amplitude of the noise is small enough).

Coming back to the region of parameters defined by (H3), we have shown that the return map of a network containing inhibitory neurons, satisfies a general definition of piecewise contractive map (Definition \ref{DEFCONTRACTIVEMAP}). To analyze the asymptotic dynamics of such maps, we define two complementary sets of the phase  space: the stable set and the sensitive set. The stable points fulfill a criterion of stability which is weaker than the negativeness of all the Lyapunov exponents. Our criterion is well defined even for maps which are neither differentiable nor continuous, and it is satisfied in particular by the maps with negative maximal Lyapunov exponent. On the other hand, the orbit of a sensitive point can be drastically modified by some arbitrarily  small perturbations which change the firing neurons. The sensitive set contains all the points whose orbit gets arbitrarily near to the discontinuities set of the return map.

We show that the stable set of a piecewise contractive map contains a countable number (possibly infinite) of limit cycles such that the union of their basins of attraction contains all the stable points. It results that the asymptotic dynamics of a network with a stable initial state is periodic.
The limit cycles are persistent under deterministic and stochastic perturbations, provided that their amplitude is small enough. Nevertheless, not all the limit cycles have the same stability, in particular, those that are near to sensitive points, can be destroyed by small, but non infinitesimal, perturbations. An orbit approaching them can therefore suffer a drastic change under the influence of a small perturbation.
This perturbation typically changes the basin of attraction the orbits belongs to, and it enlarges the transient time necessary to reach a limit cycle. Moreover, these ``weakly" stable cycles have the longest transient times to approach them and the largest periods. In the extreme case where the stable set contains an infinite number of limit cycles, an infinity of them are concentrated near to sensitive points. It results that these limit cycles attract a large region of the phase space, and that their experimental observation needs high precision and large time scales. The existence of very long disordered transients in a stable dynamics has been described as stable chaos in \cite{PT10} or virtual chaos in \cite{C08} and has been observed in numerical simulations of neural networks in \cite{C08}, \cite{ZBH09} and \cite{Z06}, for instance.

On the other hand, the characterization of all the possible attractors of the sensitive set is still an open question. Nevertheless, it is known that for one dimensional piecewise contractive maps the attractor of the sensitive points is essentially a Cantor set. Also, in \cite{CGMU} it has been reported that in higher dimension, other kind of attractors with a cardinality or a complexity lower or higher than the one of a Cantor set can exist. None of the examples described in \cite{CGMU} contains periodic orbits. We can therefore expect that the asymptotic periodic dynamics of stable IF neural networks can coexist with a non periodic sensitive dynamics. The coexistence of stable and sensitive dynamics in the same system, as well as the abundance of stable chaos, depend on how intricate are the sensitive set and the stable set and on their respective size.

It is mainly admitted that the existence of strong chaos needs expansivity. We have shown that IF neural networks can have such a property. Our main purpose was to show that some conditions were needed for the return map to be piecewise contractive. However, it would be interesting to make a deeper study of the phenomenon of non global contraction or expansivity, in some region of the phase space, as in Theorem \ref{SNOCONT}. Although our result is stated in a different way, its proof relies on the existence of a repulsive periodic orbit. We think that it is possible to find more of these orbits, which would be a further step in the research of chaos in neural networks.

To sum up, in the present paper we give a rigorous analysis of the asymptotic dynamics of a family of
IF neural networks. Our analysis is principally qualitative and proved under well posed hypothesis, which can certainly be relaxed. In particular, the results about the stable dynamics of the networks under study are proved in some region of the parameters, where the return map is a piecewise contraction. But these results apply to any piecewise contractive map satisfying the general Definition \ref{DEFCONTRACTIVEMAP}. Therefore, Proposition \ref{CYCLESETA} and Theorem \ref{teoremaPrincipal}, as well as their consequences, will then be true for any neural network (possibly more sophisticated) in the parameters region where its return map satisfies Definition \ref{DEFCONTRACTIVEMAP}. Also, the proofs of the principal theorems and propositions give some hints to obtain more quantitative results. We hope that the results of the present paper have helped to identify some important concepts and to pose some basis of future studies.

\section{Appendix: Some technical proofs.}

\subsection{Proof of Lemma \ref{PARAM}} \label{PROOFPARAM}

Let $i\neq j \in I$. To prove the Lemma \ref{PARAM}, we suppose the neurons $i$ and $j$ satisfy  the four  conditions (O1), (O2) and (O3) stated before the lemma.

\vspace{1ex}

\noindent {\bf Step 1}: We show that $\Gamma_i\subset\Sigma_{I\setminus\{j\}}$, that is to say $J(\V)=I\setminus\{j\}$ (to prove $\Gamma_j\subset\Sigma_{I\setminus\{i\}}$ just permute $i$ and $j$). Suppose $\V\in\Gamma_{i}$. Then, $V_k<V_i$ for all $k\neq i$ and the neuron $i$ is the only neuron which reaches the threshold spontaneously at time $t_i(\V)=\bar{t}(\V)$. It implies that $\V\in\Sigma_i$ and $J_{0}(\V)=\{i\}$ (recall (\ref{PARTITIONBI})).

 Now let us compute the set of the neurons that fire by interaction with the neuron $i$ at time $\bar{t}(\V)$. Let $k\in I\setminus J_0(\V)=I\setminus\{i\}$, then using (\ref{PHITV}) with $V_k=0$ we obtain
% Just before the firing regime, the potential of the neuron $k$ is $\phi^{t_i(\V)}_k(\V)$ and the interaction with neuron $i$ leads to:
\[
\phi^{\bar{t}(\V)}_k(\V) + \sum_{l\in J_{0}(\V)\,:\, H_{lk}>0}H_{lk}=\phi^{t_i(\V)}_k(\V) + H_{ik}=\beta-
\frac{\beta(\beta-\theta)}{\beta-V_i} + H_{ik},
\]
which belongs to $(H_{ik},c^\ast+H_{ik})$ since $V_i\in (c^\ast,\theta)$. If $k\neq j$, then $k\notin\{i,j\}$ and by hypothesis (O2) we have $H_{ik}>\theta$. Therefore, $k\in J_1(\V)$.
If $k=j$, by hypothesis (O1) we have $H_{ik}<c^\ast-\theta$ and $k\notin J_1(\V)$. We deduce that $J_1(\V)=I
\setminus \{j\}$.

Now let us compute the set of the neurons that fire by interaction with the neuron of $J_1(V)$ at time $\bar{t}(\V)$. Let $k\in I\setminus J_1(\V)=\{j\}$, then using (\ref{PHITV}) with $V_j=0$ and (O1) we obtain
\[
\phi^{\bar{t}(\V)}_k(\V) + \sum_{l\in J_{1}(\V)\,:\, H_{lk}>0}H_{lk}=\phi^{t_i(\V)}_j(\V) + \sum_{l\neq j \,:\,
H_{lj}>0} H_{lj} \
<c^\ast+\theta-c^\ast=\theta.
\]
It follows that $j\notin J_2(\V)$. We deduce that $J_2(\V)=J_1(\V)=I\setminus\{j\}$ and then
$J(\V)=I\setminus\{j\}$.

\vspace{1ex}
\noindent{\bf Step 2}: We show that there exists $(a,b)\subset (c^\ast,\theta)$ such that for all
$\V\in\Gamma_{i}$ verifying $V_i\in(a,b)$ we have $\rho(\V)\in\Gamma_{j}$. Suppose $\V\in\Gamma_{i}$. From step 1, we have $\Gamma_{i}\subset\Sigma_{I\setminus\{j\}}$, which implies $\rho_k(\V)=0$ for all $k\neq j$. Remains to show that for some $c^\ast<a<b<\theta$, if $V_i\in(a,b)$ then $\rho_j(\V)\in (c^\ast,\theta)$.

Since, ${\bar t}(\V)=t_i(\V)$ using (\ref{PHITV}) we obtain
\[
\rho_j(\V)= g_j(V_i)\quad\mbox{where}\quad g_j(x):=\beta-
\frac{\beta(\beta-\theta)}{\beta-x} + \sum_{l\neq j} H_{lj}\quad\forall\,x\neq\beta.
\]
The function $g_j$ being strictly decreasing, $g_j(x)\in (g_j(\theta),g_j(c^\ast))$ for all $x\in(c^\ast, \theta)$.
By (O3) we have $g_j(c^\ast)=c^\ast+\sum_{l\neq j}H_{lj}>c^\ast$ and by (O1) we have $g_j(c^\ast)<\theta$.
It results that $g_j(c^\ast)\in(c^\ast, \theta)$. Thus, the function $g_j$ being continuous there exists
$(a,b)\subset (c^\ast, \theta)$ such that $g_j(x)\in (c^\ast,\theta)$ for all $x\in(a,b)$. It follows that
$\rho_j(\V)=g_j(V_i)\in(c^\ast,\theta)$ for all $V_i\in(a,b)$, which is the desired result.

\vspace{1ex}
\noindent Note that $g_j$ is injective in $(c^\ast,\theta)$, so $\rho$ is injective $\Gamma_i$. This ends the proof of the Lemma \ref{PARAM}.

\subsection{Proof of Theorem \ref{TMETRIC}}\label{PROOFTMETRIC}

\begin{Lemma}\label{HYPOKATOK} Under the hypothesis {\em (H1),(H2),(H3)} and {\em (H4)} there exists $c>0$ such that
for all $n\in\N$
\[
\|\rho^n(\V)-\rho^n(\W)\|\leq c\lambda^n\|\V-\W\|\qquad
\mbox{if}\qquad \V,\,\W\in\mathcal{P}_{J_0\dots J_n}:=\bigcap\limits_{i=0}^n\rho^{-i}(\Sigma_{J_i})
\]
where $J_0,\dots,J_n\in P(I)$ and $\lambda$ is the contraction constant of $\rho$ in the contractive zone
$\Sigma^*$.
\end{Lemma}

\begin{proof} Consider the integer $p$ of Proposition \ref{ATTRACT} and define for all
$k\in\{0,\dots,p\}$

\[
c_k=\max_{J_0,\dots,J_k\in P(I)}\sup\left\{\frac{\|\rho^k(\V)-\rho^k(\W)\|}{\lambda^k\|\V-\W\|},\ \V\neq\W
\in\mathcal{P}_{J_0\dots J_k}\right\}
\]
where $\mathcal{P}_{J_0\dots J_k}=\bigcap\limits_{i=0}^k\rho^{-i}(\Sigma_{J_i})$.

Let us show that $c_k$ is bounded for all $k\in\{0,\dots,p\}$. Fix a $k\in\{0,\dots,p\}$ and let
$J_0,\dots,J_k\in P(I)$. The Poincar\'e map $\rho$ being Lipchitz continuous on each $\Sigma_{J_i}$,
by construction of $\mathcal{P}_{J_0\dots J_k}$, the composition $\rho^k$ is also Lipchitz continuous on $
\mathcal{P}_{J_0\dots J_k}$. Then, there exists a constant $L>0$ such that

\[
\|\rho^k(\V)-\rho^k(\W)\|\leq L\|\V-\W\| \qquad\forall\ \V,\,\W \in\mathcal{P}_{J_0\dots J_k},
\]
and as a consequence
\[
\sup_{\V\neq\W\in\mathcal{P}_{J_0\dots J_k}}\frac{\|\rho^k(\V)-\rho^k(\W)\|}{\lambda^k\|\V-\W\|}\leq
\frac{L}{\lambda^k}.
\]
It follows that $c_k$ is bounded for all $k\in\{0,\dots,p\}$ and $c:=\max\limits_{k\in\{0,\dots,p\}} c_k$ exists.

Let $n\in\N$ and $J_0,\dots,J_n\in P(I)$. If $n\leq p$, then by definition of $c$
\begin{equation}\label{CLAMBDAN}
\|\rho^n(\V)-\rho^n(\W)\|\leq c\lambda^n\|\V-\W\|\qquad
\mbox{if}\qquad \V,\,\W\in\mathcal{P}_{J_0\dots J_n}.
\end{equation}
If $n>p$, take (\ref{CLAMBDAN}) as an induction hypothesis. Let $J_{n+1}\in P(I)$ and suppose $\V,\,\W\in
\mathcal{P}_{J_0\dots J_{n+1}}$. As $n\geq p$ by Proposition \ref{ATTRACT} and
Proposition \ref{INVARIANT} we have $\rho^n(\Sigma)\subset\Sigma^*$. Therefore, $\rho^n(\V)$ and $\rho^n(\W)\in
\Sigma^*\cap\Sigma_{J_n}$. From Proposition \ref{CONTRACTIVEZONE} it follows
\[
\|\rho^{n+1}(\V)-\rho^{n+1}(\W)\|\leq\lambda\|\rho^n(\V)-\rho^n(\W)\|
\]
and from the induction hypothesis we obtain
\[
\|\rho^{n+1}(\V)-\rho^{n+1}(\W)\|\leq c \lambda^{n+1}\|\V-\W\|.
\]
\qed
\end{proof}

Let $\lambda<\mu<1$ and let $n_0\in\N$ be such that
$c\left(\frac{\lambda_\alpha}{\mu}\right)^{n_0}<1$. Consider the metric $d$ defined by
\[
d(\V,\W):=\sum_{i=0}^{n_0-1}\frac{\|\rho^i(\V)-\rho^i(\W)\|}{\mu^i}\qquad\forall\, \V,\W\in\Sigma^*
\]
and the partition $\mathcal{P}'$ of $\Sigma$ defined by
$\mathcal{P}':=\{\mathcal{P}_{J_0\dots J_{n_0}},\  J_0\dots J_{n_0} \in P(I)\}$. To prove the theorem it is enough to show that $\rho$ is piecewise contractive in $\Sigma$ with respect to $\mathcal{P}'$ for the metric $d$. This is the purpose of the following calculation. First, note that for all $\V$, $\W$ in $\Sigma$, we have:
\begin{eqnarray*}
d(\rho(\V),\rho(\W))&=&\sum_{i=1}^{n_0}\frac{\|\rho^i(\V)-\rho^i(\W)\|}{\mu^{i-1}}\\
&=& \mu\left(d(\V,\W)+ \frac{\|\rho^{n_0}(\V)-\rho^{n_0}(\W)\|}{\mu^{n_0}} - \|\V-\W\|\right).
\end{eqnarray*}
Now, suppose $\V$ and $\W$ in $\mathcal{P}_{J_0\dots J_{n_0}}$ for some $J_0\dots J_{n_0} \in P(I)$. Then,
applying Lemma \ref{HYPOKATOK} we obtain:
\[
d(\rho(\V),\rho(\W))\leq
\mu\left(d(\V,\W)+ c\left(\frac{\lambda}{\mu}\right)^{n_0}\|\V-\W\| - \|\V-\W\|\right).
\]
By definition of $\mu$ and $n_0$, we have then
\[
d(\rho(\V),\rho(\W))\leq \mu d(\V,\W)
\]
which ends the proof of Theorem \ref{TMETRIC}.

%%%%%%%%%%%%%%%%%%%%%%%%%%%%%%%%%%%%%%%%%%%%%%%%%%%%%%%%%%%%%%%%%%%%%%%%%%%%%%%%%%%%%%%%%%%%%%%%%%%%

\vspace{.4cm}

\noindent {\bf Acknowledgements: }
We thank  R. Budelli for many suggestions and useful discussions about the biological neural networks and their dynamics. E. C. has been supported by the Agencia Nacional de Investigaci\'on e Innovaci\'on  and by Comisi\'{o}n Sectorial de Investigaci\'{o}n Cient\'{\i}fica of Universidad de la Rep\'{u}blica, both institutions of Uruguay. P. G. has been supported by project FONDECYT 1100764.

\end{document}